\documentclass[11pt]{article}
\usepackage{tocbibind} 
\usepackage{amsmath,amsthm,amssymb}
\usepackage{fancyhdr}
\usepackage[british]{babel}
\usepackage{geometry,mathtools}
\usepackage{enumitem}
\usepackage{algpseudocode}
\usepackage{dsfont}
\usepackage{centernot}
\usepackage{xstring}
\usepackage{colortbl}

\usepackage[section]{algorithm}
\usepackage{graphicx}
\usepackage[font={footnotesize}]{caption}
\usepackage[usenames,dvipsnames,table]{xcolor}
\usepackage{pstricks,tikz}
\usetikzlibrary{patterns,calc,fadings,decorations.pathreplacing}
\usepackage[h]{esvect}
\usepackage[
bookmarksopen=true,
bookmarksopenlevel=1,
colorlinks=true,
linkcolor=darkblue,
linktoc=page,
citecolor=darkblue,
urlcolor=darkblue
]{hyperref}
\usepackage{bbm}
\usepackage{cleveref}
\newcommand{\la}{\operatorname{la}}

\numberwithin{equation}{section}

\definecolor{darkblue}{rgb}{0,0,0.5}

\newdimen\margin
\def\textno#1&#2\par{
   \margin=\hsize
   \advance\margin by -4\parindent
          \setbox1=\hbox{\sl#1}
   \ifdim\wd1 < \margin
      $$\box1\eqno#2$$
   \else
      \bigbreak
      \hbox to \hsize{\indent$\vcenter{\advance\hsize by -3\parindent
      \it\noindent#1}\hfil#2$}
      \bigbreak
   \fi}

\newtheorem{theorem}[algorithm]{Theorem}
\newtheorem{prop}[algorithm]{Proposition}
\newtheorem{lemma}[algorithm]{Lemma}
\newtheorem{cor}[algorithm]{Corollary}

\theoremstyle{definition}

\newtheorem{defin}[algorithm]{Definition}
\newtheorem{remark}[algorithm]{Remark}

\newcounter{stepenv}
\newenvironment{stepenv}[1][]{\refstepcounter{stepenv}}{}

\newcounter{step}[stepenv]

\newcounter{substep}[step]
\renewcommand{\thesubstep}{\thestep.\arabic{substep}}

\newcounter{claim}
\newenvironment{claim}[1][]{\refstepcounter{claim}\par\medskip\noindent%
        \textit{Claim~\theclaim. #1} \itshape\rmfamily}{\medskip}

\newcommand{\cF}{\mathcal{F}}

\newcommand{\bE}{\mathbb{E}}

\newcommand{\bP}{\mathbb{P}}

\def\eps{{\varepsilon}}

\def\sm{\setminus}

\def\COMMENT#1{}
\def\TASK#1{}
\let\TASK=\footnote             
\let\COMMENT=\footnote          

\author{
Nemanja Dragani\'c\thanks{Bocconi University, Milan, Italy.  \\\emph{Email}: \textbf{nemanja.draganic@unibocconi.it}.
Research supported in part by the Javotte Bocconi Institute. Parts of this research were carried out while the author was at the University of Oxford and supported by SNSF grant 217926, and while the author at the EPF Lausanne and supported by SNSF grant P5R5-2\_239120.}
\and
Michael Krivelevich\thanks{
School of Mathematical Sciences, Tel Aviv University, Tel Aviv 6997801, Israel. \\
\emph{Email}: \textbf{krivelev@tauex.tau.ac.il}.
Research supported in part by NSF-BSF grant 2023688. 
}}
\title{Efficient Hamilton covers and linear arboricity of random graphs}
\begin{document} 
\maketitle
\begin{abstract}
A Hamilton cover of a graph is a collection of Hamilton cycles whose union
contains all edges. Since each Hamilton cycle covers two edges at every vertex,
every Hamilton cover has size at least \(\lceil \Delta(G)/2\rceil\). We prove
that this lower bound is tight for binomial random graphs $G(n,p)$
throughout the widest possible range of edge probabilities: if
\(\omega(n)\to\infty\) and
\[
  \frac{\log n+\log\log n+\omega(n)}{n}
  \le p=p(n) \le
  1-\frac{\omega(n)}{n^{2}},
\]
then 
$G\sim G(n,p)$ with high probability has a Hamilton cover of size
$
  \left\lceil \frac{\Delta(G)}{2}\right\rceil.
$
The main new contribution is the sparse regime near the Hamiltonicity
threshold, where we prove a conjecture of Dragani\'c, Glock, Munh\'a Correia and Sudakov. Our proof develops constructive tools for decomposing such graphs
into controlled forest systems and extending them, using reserved pseudorandom
structure, into Hamilton cycles.

We also prove the corresponding hitting-time result for the random graph process, answering a question of Hefetz, K\"uhn,
Lapinskas and Osthus.
Finally, we use our methods to show that $G\sim G(n,p)$ with high probability satisfies the celebrated Linear arboricity conjecture for every $p\leq 1$.
\end{abstract}

 \section{Introduction}

  A \emph{Hamilton cover} of a graph \(G\) is a collection of Hamilton cycles
  whose union contains all edges of \(G\). Every Hamilton cover has size at
  least
  $
  \left\lceil \frac{\Delta(G)}{2}\right\rceil ,
  $
  since a Hamilton cycle uses exactly two edges at each vertex. We call a
  Hamilton cover \emph{tight} if it has this minimum possible size. The
  question is therefore whether this local maximum-degree obstruction is the
  only obstruction.

 Hamiltonicity is a central theme in graph theory and theoretical computer science, with wide-ranging applications. Sparse graph models have received particular attention, especially sparse random graphs.
  P\'osa~\cite{posa1976hamiltonian} proved that
  the binomial random graph \(G(n,p)\) is with high probability (whp) Hamiltonian for
  \(p\ge C\log n/n\), where \(C\) is a sufficiently large constant. The
  sharp threshold, due to Koml\'os--Szemer\'edi~\cite{KS:83} and to
  Bollob\'as~\cite{bollobas1984evolution}, is
  \[
  p=\frac{\log n+\log\log n+\omega(n)}{n}
  \]
  for every \(\omega(n)\to\infty\); see, e.g.,~\cite{bollobas1998random}. The
  corresponding hitting-time theorem, due independently to Ajtai, Koml\'os and
  Szemer\'edi~\cite{AKS:85} and to Bollob\'as~\cite{bollobas1984evolution}, says
  that in the random graph process
  \(\widetilde G=(G_i)_{i=0}^{\binom{n}{2}}\), whp the first Hamilton cycle appears
  exactly when the minimum degree \(\delta(G_i)\) becomes two.

  Once Hamilton cycles appear, there are two natural robustness questions. The
  packing problem asks how many edge-disjoint Hamilton cycles one can find, and
  is governed by the minimum degree. The trivial upper bound
  \(\lfloor \delta(G)/2\rfloor\) is whp tight in \(G(n,p)\), through a sequence
  of works beginning with Bollob\'as and Frieze~\cite{bollobas1983matchings}
  and including Frieze and Krivelevich~\cite{FK:08}, Krivelevich and
  Samotij~\cite{krivelevich2012optimal}, Knox, K\"uhn and
  Osthus~\cite{knox2015edge}, and Hamilton decomposition results for regular
  expanders~\cite{kuhn2014hamilton}. The covering problem is the dual question:
  it asks how many Hamilton cycles are needed to cover all edges, and is
  governed by the maximum degree.

  Hamilton covers in random graphs were first studied systematically by Glebov,
  Krivelevich and Szab\'o~\cite{GKS:14}. They proved that for every fixed
  \(\alpha>0\), \(G(n,p)\) with \(p\ge n^{-1+\alpha}\) has a Hamilton cover of
  size \((1+o(1))np/2\) whp, and conjectured that the optimal number
  \(\lceil\Delta(G(n,p))/2\rceil\) should be attainable essentially throughout
  the Hamiltonicity range. Hefetz, K\"uhn, Lapinskas and
  Osthus~\cite{HKLO:14} proved the exact tight result for
  \[
  \frac{\log^{117} n}{n}\le p\le 1-n^{-1/8},
  \]
  already pushing the cover theorem close to the complete graph in the upper
  bound for \(p(n)\). Dragani\'c, Glock, Munh\'a Correia and
  Sudakov~\cite{DGMS:25} later lowered the sparse end to \(C_0\log n/n\), for a
  sufficiently large constant \(C_0\).

  In this paper we close the remaining gaps. In \Cref{sec:sparse-range} we prove the tight bound for the range starting at the sharp threshold, confirming the conjecture by Dragani\'c, Glock, Munh\'a Correia and
  Sudakov~\cite{DGMS:25}. In
  \Cref{sec:very-dense} we complete the remaining dense range up to the natural
  endpoint. Consequently we obtain the following ultimate form of the random graph Hamilton
  cover theorem.

\begin{theorem}\label{thm:main}
Let \(\omega(n)\to\infty\), and let \(p=p(n)\) satisfy
\[
  \frac{\log n+\log\log n+\omega(n)}{n}
  \le p\le
  1-\frac{\omega(n)}{n^{2}}.
\]
Then whp \(G\sim G(n,p)\) has a tight Hamilton cover.
\end{theorem}

  Specifically, the range from the threshold up to \(C_0\log n/n\) is the new sparse
  result proved in \Cref{sec:sparse-range}; the range from \(C_0\log n/n\) up
  to \(1-n^{-1/8}\) follows by combining~\cite{DGMS:25} with~\cite{HKLO:14};
  and the range \(p=1-q\), where \(q\le n^{-1/8}\) and \(qn^2\to\infty\), is
  covered by our \Cref{thm:very-dense-range}. Our proof is constructive, and yields a randomized polynomial time algorithm to find the required cover. In the extended version of our paper we comment on the algorithmic details more explicitly, as our proof is presented existentially for clarity of presentation. Furthermore, we comment on how the proof in~\cite{DGMS:25} and~\cite{HKLO:14} for
 the range from \(C_0\log n/n\) up
  to \(1-n^{-1/8}\) can be made algorithmic as well.
  
  The range of \(p\) in \Cref{thm:main} is best possible in the following
  sense. The lower endpoint is forced by Hamiltonicity itself: below
  \(p(n)=(\log n+\log\log n)/n\), the graph whp has a vertex of degree less than
  two, and hence has no Hamilton cycles, let alone a Hamilton cover. The upper endpoint is also the natural
  endpoint for a whp statement. If \(1-p=c/n^2\) for a fixed \(c>0\), then, for
  odd \(n\), with probability bounded away from zero the complement of
  \(G(n,p)\) consists of exactly one edge. But \(K_n\) with one edge removed has
  no tight Hamilton cover, as observed in~\cite{HKLO:14} and explained in
  \Cref{sec:very-dense}. On the other hand, if \(1-p\ll n^{-2}\), then
  \(G(n,p)=K_n\) whp, and Walecki's theorem (see, e.g.,~\cite{alspach2008walecki}) gives a tight Hamilton cover.
  Thus the only excluded part is the unavoidable critical window around
  \(1-p\asymp n^{-2}\).

  The new sparse ingredient is also robust enough to give the corresponding
  hitting-time statement in the window where the Hamiltonicity threshold occurs.
  This answers a question of Hefetz, K\"uhn, Lapinskas and Osthus~\cite{HKLO:14},
  who asked whether the tight-cover property appears at the hitting time of
  Hamiltonicity.

\begin{theorem}[Hitting time for tight Hamilton covers]
\label{thm:hitting-time-cover}
Let \(G_0\subseteq G_1\subseteq \cdots \subseteq G_{\binom n2}\) be the random
graph process on vertex set \([n]\), where \(G_i\) is obtained from
\(G_{i-1}\) by adding one edge chosen uniformly at random from the missing
edges. Let \(\tau_2=\min\{i:\delta(G_i)\ge 2\}\). Then whp \(G_{\tau_2}\) has a
tight Hamilton cover. 
\end{theorem}
Since the proof consists of the same algorithm and analysis as in the proof of \Cref{thm:main}, we briefly comment on the details in the concluding remarks in \Cref{sec:concluding}. 

Hamilton covers are also natural from an algorithmic point of view. They form a
covering analogue of Hamilton decompositions and Hamilton-cycle packings: rather
than asking for one spanning cycle, or for many edge-disjoint spanning cycles,
one asks for the minimum number of spanning tours needed to cover the whole edge
set. Closely related Hamilton decomposition questions have appeared in the
algorithmic study of TSP tour domination~\cite{GutinYeo2001,
AlonGutinKrivelevich2004}, and Ferber and Mond~\cite{ferber-mond:25} recently
proved an algorithmic optimal packing theorem for edge-disjoint Hamilton cycles
in random directed graphs with degree at least $\log^{15}n$. Our result gives an algorithmic covering counterpart for
undirected random graphs. 

  Hamilton covers also fit into a broader line of work on pseudorandom graphs.
  The result of Glebov, Krivelevich and Szab\'o~\cite{GKS:14} applies to a broad
  class of expanders, and Hamilton decompositions of regular expanders were
  obtained by K\"uhn and Osthus~\cite{kuhn2014hamilton}. More recently,
  Dragani\'c, Kim, Lee, Munh\'a Correia, Pavez-Sign\'e and
  Sudakov~\cite{DKLMPSS:25} proved asymptotically optimal packing and covering
  results for sparse pseudorandom graphs: for every \(\varepsilon>0\) there is
  \(C>0\) such that every \((n,d,\lambda)\)-graph with \(d/\lambda\ge C\)
  contains at least \((1/2-\varepsilon)d\) edge-disjoint Hamilton cycles, and
  its entire edge set can be covered by at most \((1/2+\varepsilon)d\) Hamilton
  cycles. In the directed setting, Ferber, Kronenberg and
  Long~\cite{ferber2017packing} proved asymptotic packing, counting and
  covering results for random directed graphs.

  \subsection*{Linear arboricity}

  A closely related decomposition problem is the linear arboricity conjecture of
  Akiyama, Exoo and Harary~\cite{linarbconj}. A linear forest is a graph whose
  connected components are paths. The \emph{linear arboricity} \(\la(G)\) of a
  graph \(G\) is the smallest number of linear forests whose union covers
  \(E(G)\), or equivalently, the smallest number of linear forests its edges can be partitioned into.
  The conjecture asserts that
  \[
  \la(G)\le \left\lceil \frac{\Delta(G)+1}{2}\right\rceil
  \]
  for every graph \(G\). This would be best possible, for example for regular
  graphs of odd degree, and can be viewed as a path-like strengthening of the
  basic degree barrier in edge decompositions. The conjecture remains open in
  general. Alon~\cite{alon1988linear} proved the asymptotic form
  \(\la(G)\le (1/2+o(1))\Delta(G)\); later work of Ferber, Fox and
  Jain~\cite{FFJ:20} and Lang and Postle~\cite{LangPostle} improved the error
  term, which is now of order $\Tilde{O}(\sqrt{\Delta(G)})$. Very recently, Christoph, Dragani\'c, Gir\~ao, Hurley, Michel and
  M\"uyesser~\cite{CDGHMM:25} proved that every \(n\)-vertex graph satisfies
  \(\la(G)\le \Delta(G)/2+O(\log n)\), improving the bound for graphs with at least polylogarithmic degree. For random graphs, McDiarmid and
  Reed~\cite{arboricity-random-regular} proved the conjecture for random
  regular graphs, and Glock, K\"uhn and Osthus~\cite{GKO:16} proved it for a
  large range of binomial random graphs using the Hamilton-cover result of
  Hefetz, K\"uhn, Lapinskas and Osthus~\cite{HKLO:14}, and the range was later extended by the Hamilton cover result in~\cite{DGMS:25}. Our Hamilton-cover theorem, together
  with the additional argument for sparse graphs in \Cref{sec:linear-arboricity},
  yields the full binomial random graph result.

  \begin{theorem}\label{thm:random-linear-arboricity}
  For every \(p=p(n)\), whp \(G\sim G(n,p)\) satisfies
  \[
  \la(G)\leq \left\lceil\frac{\Delta(G)+1}{2}\right\rceil .
  \]
  \end{theorem}

In fact, one can see from the proof that a stronger bound of $\lceil \Delta(G)/2\rceil$ holds for every $p$ with $1-p=\omega(1/n^2)$.

\subsection{Proof outline for the sparse range}
We finish the introduction by outlining the new sparse-range argument in some
more detail. The difficulty in this regime is that the graph is already
Hamiltonian, but it is still far from regular in the sense relevant to a tight
cover. Whp one has \(\Delta(G)-\delta(G)=\Theta(\Delta(G))=\Theta(\log n)\).
A tight cover contains only \(k:=\lceil\Delta(G)/2\rceil\) Hamilton cycles, so
at a vertex of degree close to \(\Delta(G)\) almost every available incident edge
must be used without repetition, while at a vertex of very small degree the same few
incident edges must be reused many times by the \(k\) cycles. Thus the
cycles have to be routed through the exceptional vertices in a very controlled way.

The first step is to isolate the exceptional vertices. We fix a small constant
\(\alpha>0\), depending only on the constant \(C\) in the upper bound
\(p\le C\log n/n\). Let \(B\) be the set of vertices of degree at least
\((1-\alpha)\Delta(G)\), let \(S\) be the set of vertices of degree less than
\((\log\log n)^4\), and put \(W=B\cup S\cup N(S)\). Standard first-moment
estimates show that these sets are typically small and well separated: there are no short
paths between exceptional vertices, and no short cycles in their
neighbourhoods. This separation is a useful feature of the random graph near
the threshold. It allows us to deal with the high-degree and low-degree
constraints locally, while leaving the typical part of the graph available for
global routing.

We then cover the typical part \(G-W\) by linear forests, using the approach in~\cite{DGMS:25}. The reason for using
linear forests is that they are the right intermediate objects: every linear
forest can potentially be completed to a Hamilton cycle, provided it has access
to enough unused random edges and vertices. To keep this access available, we randomly split
\(G-W\) into a fixed constant number of subgraphs, where the number of parts
depends only on \(C\). Each part comes with its own linear-sized reservoir of
vertices outside it. The split is chosen so that every part has maximum degree
noticeably smaller than \(\Delta(G)/k\) (recalling that $\Delta(G-W)\leq (1-\alpha)\Delta(G)$), and every vertex still
has many neighbours in each corresponding reservoir. Applying the approximate
linear arboricity theorem separately to these parts gives a collection of
linear forests covering \(G-W\). Crucially, this collection has size
smaller than \(k\) by a positive fraction of \(k\). The unused fraction is the
slack that will later absorb the exceptional vertices and the few edges lost
when the typical and exceptional pieces are merged.

The vertices in \(B\) and \(S\) are handled by explicit local constructions. For
each high-degree vertex \(b\in B\), we distribute the edges incident to \(b\)
among \(k\) cherries with centers at that vertex, so that edges are repeated only when necessary.  For each
low-degree vertex \(s\in S\), the situation is different: every Hamilton cycle
must pass through \(s\), but there are very few edges available there. We build
\(k\) small forest gadgets through \(\{s\}\cup N(s)\), with endpoints pushed out
to \(N^2(s)\), so that all edges incident to the low-degree region are covered.
Because the vertices in \(S\) are far apart, these gadgets do not interfere with
one another. Combining the high-degree and low-degree gadgets gives \(k\)
exceptional forests.

At this point we have two kinds of forests: typical forests covering \(G-W\),
and exceptional forests covering the edges incident to \(B\), \(S\), and
\(N(S)\). The next task is to combine them without increasing the total number
of forests. We pair each typical forest with one exceptional forest. If the two
forests conflict at a vertex, we delete the offending typical edge before taking
their union. The deleted edges form a graph of bounded maximum degree, a
consequence of the separation of \(B\) and \(S\) and the fact that vertices
outside \(W\) have small neighbourhood in $W$. The
exceptional forests that were not paired provide enough spare capacity to absorb
this bounded-degree leftover graph by a simple merging lemma. Thus we obtain
exactly \(k\) linear forests covering all edges of \(G\), and each forest still
has a private linear-sized reservoir in which all of its endpoints have many
neighbours.

The final step is to extend each of these $k$ forests to a Hamilton cycle. Conceptually, the task is to route the endpoints of the forests through the remaining random set. This is the point where the sparse regime creates a genuine technical difficulty: the available reservoir is large, but it comes with only a weak minimum-degree guarantee.
Fortunately, one can show that for carefully chosen parameters, one can utilize the extendability method, introduced in \Cref{sec:FP}, which allows us to connect the components of each forest using the reserved random edges and then close the resulting path into a Hamilton cycle. Applying this extension lemma to all $k$ forests gives the desired tight Hamilton cover. The procedure is constructive throughout; on the high-probability event given by the random graph estimates, all steps can be implemented in randomized polynomial time.


\section{Preliminaries}
Here we state and show several basic graph theoretic results, as well as some probabilistic tools. We first introduce some standard notation.

\subsection{Notation and definitions}
We use standard graph theoretic notation. Given a graph $G$, we denote by $V(G)$ its vertex set and by $E(G)$ its edge set. Given $S\subseteq V(G)$, we denote by $G[S]$ the subgraph of $G$ induced by $S$, and by $N_G(S)$ the external neighbourhood of $S$ in $G$ (omitting the subscript whenever it is clear from the context). We denote by $N^k_G(S)$ the set of vertices at distance $k$ from $S$, in particular $N^1_G(S)=N_G(S)$. For subsets $S_1,S_2\subseteq V(G)$ we denote by $e_G(S_1,S_2)$ the number of edges of $G$ with one endpoint in $S_1$ and the other in $S_2$. For a vertex $v\in V(G)$, we denote by $\partial_G(v)$ the set of edges incident to $v$ in $G$. We denote by $\delta(G)$ and $\Delta(G)$ the minimum and maximum degrees of $G$, respectively. A \emph{linear forest} is a graph consisting of a disjoint union of paths. For a linear forest $F$, we let $End(F)$ denote the set of vertices of degree exactly $1$ in $F$ (that is, the path endpoints, with isolated vertices excluded), and for a collection $\cF$ of linear forests we set $End(\cF):=\bigcup_{F\in\cF}End(F)$. A Hamilton cycle in a graph $G$ is a cycle containing every vertex in $G$. We let $G(n, p)$ denote the binomial random graph on $n$ labeled vertices, where each edge is included independently with probability $p$. All logarithms are natural unless stated otherwise.
\subsection{Auxiliary results about linear forests}
We start with a simple lemma which allows us to extend a linear forest by edges of a low maximum degree graph.
  \begin{lemma}\label{lem:merginglinfor}
  Let \(H_1,H_2\) be graphs, and let \(\mathcal F\) be a collection of linear
  forests covering \(E(H_1)\). Suppose that every vertex of \(V(H_2)\) has
  degree
  at most \(d\) in both \(H_1\) and \(H_2\), and that every edge of \(H_1\)
  incident
  to \(V(H_2)\) lies in at most \(\ell\) forests of \(\mathcal F\). If
  \[
  |\mathcal F|\ge 4d\ell+1,
  \]
  then the edges of \(H_2\) can be added to the forests in \(\mathcal F\) so
  that
  the resulting collection is again a collection of \(|\mathcal F|\) linear
  forests covering \(E(H_1)\cup E(H_2)\). Moreover, each resulting forest is
  obtained from its original forest by adding edges of \(H_2\) only, and the
  union
  of endpoint sets is contained in \(End(\mathcal F)\cup V(H_2)\).
  \end{lemma}

\begin{proof}
We add the edges of $H_2$ one by one to linear forests in $\mathcal{F}$. Suppose we have added $i<e(H_2)$ edges from $H_2$ to linear forests in $\mathcal{F}$, and consider an edge $e=(u,v)$ of $H_2$ which we have not added yet. The endpoints $u,v$ lie in $V(H_2)$, so they each have degree at most $d$ in $H_1$ and at most $d$ in $H_2$, hence degree at most $2d$ in $H_1\cup H_2$. Therefore the number of forests of $\mathcal{F}$ in which $u$ is not isolated is at most $2d\ell$, and the same holds for $v$. Since $|\mathcal{F}|\geq 4d\ell+1$, there is a forest of $\mathcal{F}$ in which both $u$ and $v$ are isolated, so we can add $e$ to that forest. This completes the proof.
\end{proof}

\noindent We also need an approximate version of the linear arboricity conjecture, first shown by Alon~\cite{alon1988linear}.
\begin{theorem}\label{thm:approxLAC}
For every $\eps>0$ there exists $\Delta_0$ such that the following holds for all $\Delta\ge \Delta_0$. Every graph with maximum degree at most $\Delta$ can be decomposed into $\lceil(1+\eps)\Delta/2\rceil$ linear forests.
\end{theorem}
\subsection{Degree sequence in random graphs}\label{sec:highdegree}
\noindent In this section we will discuss various properties typically satisfied by vertices of high degree, that is, close to the maximum degree, in the random graph $G(n,p)$. 
\begin{lemma}[\cite{bollobas1998random}]\label{lem:binomestimate}
Let $X\sim \mathrm{Bin}(n,p)$ with $pn \geq 1$ and $q := 1-p$. Then, if $hqn \geq 3$, we have
$$\mathbb{P} \left(X \geq pn + h \right) < \sqrt{\frac{pqn}{2h^2 \pi}} \cdot \exp\left(-\frac{h^2}{2pqn} + \frac{h}{pqn} + \frac{h^3}{p^2n^2} \right) .$$
\end{lemma}

\begin{lemma}[\cite{bollobas1998random},\cite{krivelevich2012optimal}]\label{lem:maxdegree}\label{lem:degrees in gnp}
Let $\log n/n\leq p\leq n^{-1/2}$. With high probability, the minimum degree $\delta$ and the maximum degree $\Delta$ of $G(n,p)$ satisfy
\begin{enumerate}[label=\rm{(\roman*)}]
\item $pn+(1-o(1))\sqrt{2pn\log n}\leq \Delta\leq pn+2\sqrt{2pn\log n}$;
\item $\delta\leq pn-\frac{1}{4}\sqrt{2pn\log n}$.

\end{enumerate}
\end{lemma}
\begin{proof}
For the first part, the upper bound follows by a union bound over all vertices, since by Chernoff bounds the probability that a vertex has degree at least  $pn+2\sqrt{2pn\log n}$ is at most $o(1/n)$.
The lower bound is a direct consequence of Theorem 3.12 in \cite{bollobas1998random} with $m$ being a function which tends to infinity arbitrarily slowly and noting that $\Delta=d_1\geq d_m$. The second part is proven in (\cite{krivelevich2012optimal}, Lemma 2.2).
\end{proof}

\begin{prop}\label{prop:bad vertices}
Let $C>1$, consider $\log n/n\leq p\leq C\log n/n$, and let $\alpha=1/10$. Define $B$ to be the set of vertices with degree at least $pn + (1-\alpha)\sqrt{2pn\log n}$ in $G(n,p)$, and let $S$ be the vertices with degree less than $(\log\log n)^4$. Then whp all of the following hold.
\begin{enumerate}
    \item $|B|,|S| \leq \sqrt{n}$.
    \item There are no paths of length between $1$ and $6$ with both endpoints in $B\cup S$. There are also no cycles of length at most $6$ that contain a vertex in $B\cup S\cup N(S)$.
    \item If $np\leq 1.01\log n$ then $\log n/10\leq d(v)\leq 1.1\log n$ for all $v\in N(S)$.
    \item If $np\geq 1.01\log n$ then $\Delta/20\leq d(v)\leq (1-\varepsilon)\Delta$ for all $v\in N(S)$, where $\varepsilon=\varepsilon(C)>0$ is small enough.
\end{enumerate}
\end{prop}
\begin{proof}
We prove each item using first moment arguments.

\medskip\noindent\textit{Item 1.}
Each vertex $v$ has degree $d(v)\sim \mathrm{Bin}(n-1,p)$ with mean $\mu:=(n-1)p$. For $B$, write $\mu=(n-1)p$ and
  $h=(1-\alpha)\sqrt{2pn\log n}$. By the Chernoff bound (see, e.g., \cite[Theorem~2.1]{janson2000random}),
  \[
  \bP[d(v)\geq \mu+h]
   \leq \exp\left(-\frac{h^{2}}{2(\mu+h/3)}\right).
  \]
  Since $np\geq \log n$ and $\alpha=1/10$, we have
  \[
  \frac{h}{\mu}\leq (0.9+o(1))\sqrt{2},
  \]
  and therefore
  \[
  \frac{h^{2}}{2(\mu+h/3)}
  \geq
  \frac{0.9^{2}\cdot 2pn\log n}
       {2pn\left(1+0.9\sqrt{2}/3+o(1)\right)}
  =
  \frac{0.81}{1+0.3\sqrt{2}+o(1)}\log n
  >0.56\log n.
  \]
  Thus
  \[
  \bP[v\in B]\leq n^{-0.56+o(1)}
  \quad\text{and hence}\quad
  \bE[|B|]\leq n^{0.44+o(1)}=o(\sqrt n).
  \]

For $S$, since $(\log\log n)^4\ll \mu$, the standard Poisson-type bound gives
$$\bP[v\in S]\leq \left(\frac{e\mu}{(\log\log n)^4}\right)^{(\log\log n)^4}\cdot e^{-\mu}\leq \exp\left(-\mu+O((\log\log n)^5\right)\leq \exp\left(-\log n+O((\log\log n)^5)\right).$$
Hence $\bE[|S|]\leq \exp(O((\log\log n)^5))=o(\sqrt n)$. By Markov's inequality, $|B|,|S|\leq \sqrt{n}$ whp.

\medskip\noindent\textit{Item 2.}
We use the first moment method. For a specific ordered tuple $(v_0,v_1,\ldots,v_k)$ with $k\leq 6$, the probability that $v_0v_1\cdots v_k$ is a path with $v_0,v_k\in B\cup S$ is at most $O(p^k\cdot \bP[v_0\in B\cup S]\cdot\bP[v_k\in B\cup S])$, since conditioning on the path edges incident to each endpoint changes degrees by at most $O(1)$, negligibly affecting the tail probabilities. Summing over all tuples, the expected number of such paths is at most
$$\sum_{k=1}^{6}O(n(np)^k)\cdot \bP[v\in B\cup S]^2\leq O(n(np)^6)\cdot n^{-1.18+o(1)}=O(n^{-0.18+o(1)}(\log n)^6)=o(1),$$
where we used $np=O(\log n)$ and $\bP[v\in B]\leq n^{-0.59+o(1)}$ (the contribution from $S$ is dominated). Similarly, the expected number of cycles of length $k\leq 6$ containing a vertex from $B\cup S$ is at most
$$\sum_{k=3}^{6}O((np)^k)\cdot \bP[v\in B\cup S]\leq O((\log n)^6\cdot n^{-0.59+o(1)})=o(1).$$
For cycles of length $k\leq 6$ containing a vertex of $N(S)$: conditioned on a fixed cycle being present, the probability that a specified cycle vertex has an additional neighbour in $S$ (and therefore lies in $N(S)$) is at most $np\cdot \bP[v'\in S]=O(\log n\cdot n^{-1+o(1)})$, so the expected number of such cycles is at most
$$\sum_{k=3}^{6}O((np)^k)\cdot O(\log n\cdot n^{-1+o(1)})=O((\log n)^7\cdot n^{-1+o(1)})=o(1).$$
By the first moment method, whp no such paths or cycles exist.

\medskip\noindent\textit{Item 3.}
Suppose $np\leq 1.01\log n$. For a pair of vertices $u,v$ with $uv\in E(G)$, conditioned on the edge $uv$, the residual degrees $d(u)-1$ and $d(v)-1$ are independent $\mathrm{Bin}(n-2,p)$ random variables. The expected number of edges $uv$ with $d(u)<(\log\log n)^4$ and $d(v)>1.1\log n$ is at most
$$n^2 p\cdot \bP[\mathrm{Bin}(n-2,p)<(\log\log n)^4]\cdot \bP[\mathrm{Bin}(n-2,p)>1.1\log n-1].$$
The first probability is $\exp(-\log n+O((\log\log n)^5))$ as before. For the second one, since $(n-2)p\leq 1.01\log n$, the threshold exceeds the mean by a factor of $1.1/1.01>1.08$, and the Chernoff bound gives $\exp(-\Omega(\log n))=n^{-\Omega(1)}$. The product is $n^2p\cdot n^{-1+o(1)}\cdot n^{-\Omega(1)}=o(1)$, since $(\log\log n)^5=o(\log n)$ absorbs the subpolynomial factor. The same argument applies for the lower bound: the expected number of edges $uv$ with $d(u)<(\log\log n)^4$ and $d(v)<\log n/10$ is at most
$$n^2p\cdot \exp(-\log n+O((\log\log n)^5))\cdot \exp(-\Omega(np))=o(1),$$
using $\bP[\mathrm{Bin}(n-2,p)<\log n/10]\leq \exp(-\Omega(np))$ by Chernoff (the threshold is a small fraction of the mean). Thus whp $\log n/10\leq d(v)\leq 1.1\log n$ for all $v\in N(S)$.

\medskip\noindent\textit{Item 4.}
By Item~2, every vertex $v\in N(S)$ satisfies $v\notin B$, so $d(v)<np+(1-\alpha)\sqrt{2np\log n}$. Writing $r=\sqrt{2\log n/(np)}$, by \Cref{lem:maxdegree} we have $\Delta\geq np+(1-o(1))\sqrt{2np\log n}=np(1+(1-o(1))r)$, and thus
$$\frac{d(v)}{\Delta}< \frac{1+(1-\alpha)r}{1+(1-o(1))r}\leq 1-\frac{(\alpha-o(1))r}{1+r}\leq 1-\varepsilon,$$
where $\varepsilon=\varepsilon(C)>0$, since $np\leq C\log n$ gives $r\geq \sqrt{2/C}\cdot(1+o(1))$, and thus $\varepsilon\geq \alpha\sqrt{2/C}/(2(1+\sqrt{2/C}))-o(1)>0$. For the lower bound, condition on the high-probability event from \Cref{lem:maxdegree} that $\Delta\leq np+2\sqrt{2np\log n}\leq 4np$ (using $np\geq 1.01\log n$), so $\Delta/20\leq np/5$. The same argument as in Item~3 then shows the expected number of edges $uv$ from $S$ to a vertex of degree less than $\Delta/20\leq np/5$ is at most
$$n^2p\cdot \exp(-\log n+O((\log\log n)^5))\cdot \bP[\mathrm{Bin}(n-2,p)<np/5]= o(1),$$
since the Chernoff lower-tail bound gives $\bP[\mathrm{Bin}(n-2,p)<np/5]\leq \exp(-\Omega(np))$. Thus whp $\Delta/20\leq d(v)\leq (1-\varepsilon)\Delta$ for all $v\in N(S)$.
\end{proof}

\begin{remark}\label{rem:bad vertices low degree}
  The second property above implies that every vertex
  \[
  x\in V(G)\setminus (B\cup S\cup N(S))
  \]
  satisfies
  \[
  d_G\bigl(x,\,B\cup N(B)\cup S\cup N(S)\cup N^2(S)\bigr)\leq 1.
  \]
  Indeed, if \(x\) had two neighbours in
  \(B\cup N(B)\cup S\cup N(S)\cup N^2(S)\), then, since
  \(x\notin B\cup S\cup N(S)\), each of these neighbours can be joined to a
  witness in \(B\cup S\) by a distinct path of length at most \(2\). Together with \(x\), these two witness paths
  produce either a path of length at most \(6\) with both endpoints in
  \(B\cup S\), or a cycle of length at most \(6\) containing a vertex of
  \(B\cup S\cup N(S)\), contradicting Item~\(2\).
  \end{remark}





\section{Extending a linear forest into a Hamilton cycle}

In this section we show how an appropriate linear forest in a random graph can
be extended into a Hamilton cycle. We start with some basic properties
typically satisfied by random graphs. 

\begin{defin}\label{def:C-expander}
For a real $C\geq 2$, a graph $H$ is a \emph{$C$-expander} if
\begin{enumerate}[label=\rm{(\roman*)}]
\item every set $X\subseteq V(H)$ with $|X|\leq |V(H)|/(2C)$ satisfies $|N_{H}(X)|\geq C|X|$; and
\item every two disjoint sets $A,B\subseteq V(H)$ with $|A|,|B|\geq |V(H)|/(2C)$ have at least one edge of $H$ between them.
\end{enumerate}
\end{defin}

\begin{lemma}\label{lem:gnp properties}
Let $G\sim G(n,p)$ and fix $\varepsilon>0$ and a constant $c_{0}\geq 2$. If $\log n/n \leq p \leq c_{0}\log n/n$,
then whp $G$ has the following properties.
\begin{enumerate}[label=\rm{(\alph*)}]
    \item\label{p:alpha joint}
    For all integers $a,b\geq 1$ with $abp\geq 10n$, every pair of disjoint
    sets $A,B\subseteq V(G)$ with $|A|=a$ and $|B|=b$ have at least one edge between them.
    \item\label{p:C-expander}
    Let $S\subseteq V(G)$ with $|S|\geq \varepsilon n$ and suppose
    $d_S(v)\geq (\log\log n)^{3}$ for every $v\in S$. Then $G[S]$ is a
    $\log\log n$-expander.
\end{enumerate}
\end{lemma}

\begin{proof}
For \ref{p:alpha joint}, fix integers $a,b\geq 1$ and disjoint $A,B\subseteq V(G)$
of sizes $a,b$. Then
$\mathbb{P}[e(A,B)=0]=(1-p)^{ab}\leq e^{-pab}\leq e^{-10n}$.
The number of choices of $(a,b,A,B)$ is at most
$n^{2}\binom{n}{a}\binom{n}{b}\leq n^{2}\cdot 2^{2n}$,
and $n^{2}\cdot 2^{2n}\cdot e^{-10n}=o(1)$.

For \ref{p:C-expander}, we show that every $X\subseteq S$ with
$|X|\leq |S|/(2\log\log n)$ satisfies
$|N_{G[S]}(X)|\geq |X|\log\log n$.
We distinguish three cases.

\medskip\noindent\emph{Small sets: $|X|\leq (\log\log n)^{2}/2$.}
Pick any $v\in X$. By the minimum-degree assumption, $v$ has at least
$(\log\log n)^{3}$ neighbours in $S$, of which at most $|X|$ lie in $X$. Hence
\[
|N_{G[S]}(X)|\geq (\log\log n)^{3}-|X|\geq (\log\log n)^{3}/2\geq |X|\log\log n.
\]

\medskip\noindent\emph{Medium sets: $(\log\log n)^{2}/2<|X|\leq
\tfrac{100n}{\varepsilon\log n}$.}
We first show that whp every $Y\subseteq V(G)$ with
$|Y|\leq y_{0}:=\tfrac{300\,n\log\log n}{\varepsilon\log n}$
spans fewer than $\alpha|Y|$ edges in $G$, where
$\alpha:=(\log\log n)^{2}/8$.

By a union bound, the probability that some such $Y$ has at least $\alpha|Y|$
edges is at most
\begin{align*}
\sum_{2\leq y\leq y_{0}}\binom{n}{y}\binom{\binom{y}{2}}{\alpha y}p^{\alpha y}
\;\leq\;\sum_{2\leq y\leq y_{0}}\left(\frac{en}{y}\right)^{y}
\left(\frac{eyp}{2\alpha}\right)^{\alpha y}
\;=\;\sum_{2\leq y\leq y_{0}}B(y)^{y},
\end{align*}
where $B(y):=\dfrac{en}{y}\left(\dfrac{eyp}{2\alpha}\right)^{\alpha}$.
Factoring out the $y$-dependence,
\[
B(y)=\bigl[en\cdot(ep/(2\alpha))^{\alpha}\bigr]\cdot y^{\alpha-1},
\]
and since $\alpha>1$ for $n$ large, $B$ is increasing in $y$, so
$B(y)\leq B(y_{0})$ for all $y\in[2,y_{0}]$. Plugging in $y=y_{0}$ and using
$p\leq c_{0}\log n/n$,
\[
\frac{ey_{0}\,p}{2\alpha}\leq \frac{1200\,c_{0}\,e}{\varepsilon\log\log n},
\]
so we have
\[
B(y_{0})\;\leq\;\frac{e\varepsilon\log n}{300\log\log n}
\left(\frac{1200\,c_{0}\,e}{\varepsilon\log\log n}\right)^{(\log\log n)^{2}/8}=o(1).
\]
In particular $B(y_{0})<1/2$ for $n$ large, hence
\[
\sum_{2\leq y\leq y_{0}}B(y)^{y}\;\leq\;\sum_{y\geq 2}B(y_{0})^{y}
\;\leq\;\frac{B(y_{0})^{2}}{1-B(y_{0})}=o(1),
\]
proving the claim.

Now suppose for contradiction that some $X\subseteq S$ in the medium range
satisfies $|N_{G[S]}(X)|<|X|\log\log n$. Let $Y:=X\cup N_{G[S]}(X)$, so
\[
|Y|\leq |X|(1+\log\log n)\leq 2|X|\log\log n\leq
\frac{200\,n\log\log n}{\varepsilon\log n}\leq y_{0}.
\]
By the minimum-degree assumption,
$\sum_{v\in X}d_{S}(v)\geq |X|(\log\log n)^{3}$.
On the other hand, every edge incident to a vertex of $X$ with the other
endpoint in $S$ has that endpoint in $N_{G[S]}(X)\subseteq Y$, so
\[
\sum_{v\in X}d_{S}(v)
\;\leq\;2e_{G}(Y)
\;<\;2\alpha|Y|
\;\leq\;\frac{(\log\log n)^{2}}{4}\cdot 2|X|\log\log n
\;=\;\frac{|X|(\log\log n)^{3}}{2},
\]
a contradiction.

\medskip\noindent\emph{Large sets:
$\tfrac{100n}{\varepsilon\log n}<|X|\leq |S|/(2\log\log n)$.}
Suppose $|N_{G[S]}(X)|<|X|\log\log n$, and set
$B:=S\setminus(X\cup N_{G[S]}(X))$, so that $e_{G}(X,B)=0$ and, for $n$ large,
$|B|\geq |S|-|X|(1+\log\log n)\geq |S|-\frac{|S|(1+\log\log n)}{2\log\log n}\geq |S|/3\geq \varepsilon n/3$. Then
\[
|X|\cdot|B|\cdot p
\;\geq\;\frac{100n}{\varepsilon\log n}\cdot\frac{\varepsilon n}{3}\cdot\frac{\log n}{n}
\;=\;\frac{100n}{3}\;\geq\;10n,
\]
contradicting \ref{p:alpha joint}.

\medskip\noindent\emph{Connectivity property.} It remains to verify that every two disjoint
$A,B\subseteq S$ with $|A|,|B|\geq |S|/(2\log\log n)$ have at least one edge in $G$ between them.
Since $|A|,|B|\geq \varepsilon n/(2\log\log n)$,
\[
|A|\cdot|B|\cdot p
\;\geq\;\left(\frac{\varepsilon n}{2\log\log n}\right)^{2}\cdot \frac{\log n}{n}
\;\geq\;10n
\]
for $n$ large, and the claim follows from~\ref{p:alpha joint}.
\end{proof}

In the next subsection we present the main tool of this section, the extendability technique.

\subsection{The Friedman--Pippenger tree embedding technique with rollbacks}
\label{sec:FP}

Building on the foundational tree-embedding results of Friedman and
Pippenger~\cite{friedman1987expanding} and Haxell~\cite{haxell2001tree},
the Friedman--Pippenger technique (also known as the \emph{extendability
method}) enables the robust construction of linear-sized trees within
expander graphs. The central result we use is~\Cref{cor:3.12}, which builds
such trees within an \emph{$m$-joined} host graph.

The technique has played a key role in resolving several long-standing problems
in graph theory; see, for instance,
\cite{draganic2022rolling, montgomery2019spanning,draganic:24}.
We will employ the framework developed in~\cite{montgomery2019spanning}, with
further details below.

We recall the notion of an $(m,D)$-extendable embedding from~\cite{montgomery2019spanning}, and the auxiliary notion of an $m$-joined graph.

\begin{defin}\label{def:m-joined}
For $m\in\mathbb{N}$ with $m\geq 1$, a graph $G$ is \emph{$m$-joined} if every two disjoint sets $A,B\subseteq V(G)$ with $|A|,|B|\geq m$ have at least one edge of $G$ between them.
\end{defin}

\begin{defin}\label{def:goodness}
Let $m,D\in\mathbb{N}$ with $D\geq 3$ and $m\geq 1$, let $G$ be a graph, and
let $S\subseteq G$ be a subgraph. We say that $S$ is $(m,D)$-\emph{extendable}
if $S$ has maximum degree at most $D$ and
\begin{equation}\label{eq:extendable}
|(N_{G}(U)\cup U)\setminus V(S)|\;\geq\;(D-1)|U|-\sum_{x\in U\cap V(S)}(d_{S}(x)-1)
\end{equation}
for every set $U\subseteq V(G)$ with $|U|\leq 2m$. (Here $N_G(U)$ denotes the external neighbourhood of $U$ in $G$, so $(N_G(U)\cup U)\setminus V(S)$ counts the vertices of $U\setminus V(S)$ together with the neighbours of $U$ lying outside $V(S)$.)
\end{defin}

\begin{cor}[Corollary~3.12 in~\cite{montgomery2019spanning}]\label{cor:3.12}
Let $D,m,\ell\in\mathbb{N}$ satisfy $m\geq 1$ and $D\geq 3$, and set
\[
k=\left\lceil \frac{\log(2m)}{\log(D-1)}\right\rceil,
\]
suppose $\ell\geq 2k+1$. Let $G$ be an $m$-joined graph, and let $S$ be an
$(m,D)$-extendable subgraph of $G$ with at most
\[
|G|-10\,Dm-(\ell-2k-1)
\]
vertices. Suppose $a,b$ are two distinct vertices of $S$, both with degree at
most $D/2$ in $S$. Then there is an $a,b$-path $P$ of length $\ell$ whose
internal vertices lie outside $V(S)$, such that $S+P$ is $(m,D)$-extendable.
\end{cor}

Before proving our main result, we recall the definition of Hamilton-connectivity
and a recent Hamiltonicity theorem for $C$-expanders.

\medskip\noindent A graph is \emph{Hamilton-connected} if for every two
vertices of the graph there is a Hamilton path between them.

\begin{theorem}[\cite{draganic:24}]\label{thm:hamilton-connected}
For a sufficiently large constant $C$, every $C$-expander is Hamilton-connected.
\end{theorem}

\medskip\noindent The following result is the key tool which allows us to extend
a given linear forest $F$ into a Hamilton cycle in a random graph, under the
condition that every vertex of the random graph has many neighbours outside
of $V(F)$.

\begin{lemma}\label{lem:linforestextension}
Fix constants $C\geq 2$ and $K>0$, let $\log n/n\leq p\leq C\log n/n$, and let $G\sim G(n,p)$. With high probability the
following holds. Suppose $F\subseteq G$ is a nonempty linear forest with no isolated vertices, and such that $S:=V(G)\setminus V(F)$ satisfies $|S|\geq Kn$ and
$|N_{S}(v)|\geq 100(\log\log n)^{3}$ for every $v\in V(G)\setminus V(F)$ and
every endpoint $v$ of a path of $F$.
Then there exists a Hamilton cycle in $G$ which covers~$E(F)$.
\end{lemma}

\begin{proof}
We assume that $G$ satisfies the properties of~\Cref{lem:gnp properties} (applied with the constant $c_{0}:=C$), together with the standard whp event $\Delta(G)=O(\log n)$ given by~\Cref{lem:maxdegree}, and
set $\varepsilon:=K^{2}/100$. Throughout, the
degree condition $|N_{S}(v)|\geq 100(\log\log n)^{3}$ will only be invoked
for $v\in S$ and for endpoints of paths in $F$.

Split $S$ uniformly at random into three parts $S_{1},S_{2},S_{3}$, each vertex
assigned independently and uniformly. For each vertex $v$ to which the hypothesis applies and each $i\in\{1,2,3\}$, let $E_{v,i}$ be the event that $|N_{G}(v)\cap S_{i}|< |N_{S}(v)|/10$. By Chernoff, $\bP[E_{v,i}]\leq \exp(-c(\log\log n)^{3})=:q$ for an absolute constant $c>0$. The event $E_{v,i}$ depends only on the random choices for the neighbours of $v$, so $E_{v,i}$ is mutually independent of all events $E_{v',j}$ with $v'$ at distance more than $2$ from $v$ in $G$; in particular, of all but at most $d:=3\Delta(G)^{2}=O(\log^{2}n)$ other events (using $\Delta(G)=O(\log n)$). Since $e\cdot q\cdot (d+1)=o(1)<1$, the Lov\'asz Local Lemma~(see Ch.~5 in~\cite{AS16}) applies and, since $1-eq\geq e^{-2eq}$, gives the quantitative bound
\[
\bP\Bigl[\bigcap_{v,i}\overline{E_{v,i}}\Bigr]\geq (1-eq)^{N}\geq \exp(-2eqN)=\exp(-O(nq)),
\]
where $N\leq 3n$ is the number of events. Separately, by Chernoff, the size event $\{|S_{i}|\geq |S|/4\geq Kn/4\text{ for every }i\}$ fails with probability at most $\exp(-\Omega(n))$. Since $nq=n\exp(-c(\log\log n)^{3})\ll n$, we have $\exp(-O(nq))\gg \exp(-\Omega(n))$, so a union bound gives that with positive probability both the LLL event and the size event occur simultaneously. Fix any such partition. In particular, $|S_{i}|\geq Kn/4$ for $i=1,2,3$.

\smallskip\noindent\textbf{Step 1: reduce to at most $\varepsilon n$ components.}
If $F$ has more than $\varepsilon n$ components, choose $T\subseteq V(F)$
containing exactly one endpoint from each path of $F$, so $|T|\geq \varepsilon n$.
Partition $T$ arbitrarily into two halves $T_{1},T_{2}$ with
$|T_{1}|,|T_{2}|\geq \varepsilon n/2$. Then
$|T_{1}|\cdot|T_{2}|\cdot p\geq (\varepsilon n/2)^{2}\cdot(\log n)/n\geq 10n$
for $n$ large, so by \ref{p:alpha joint} there is an edge between $T_{1}$ and
$T_{2}$; adding this edge to $F$ reduces the number of components by one.
Repeat until $F$ has at most $\varepsilon n$ components.

\smallskip\noindent\textbf{Step 2: reduce to at most $\tfrac{100n}{\varepsilon\log n}$
components, using $S_{1}$.}
At each step we update $F$ to a linear forest with one fewer component and one
additional vertex (drawn from $S_{1}$). Choose disjoint sets
$R_{1},R_{2}\subseteq V(F)$ of endpoints, each of size
$\tfrac{50n}{\varepsilon\log n}$, so that $R_{1}\cup R_{2}$ contains at most one endpoint of any single path of $F$ (this is possible whenever $F$ has at least $\tfrac{100n}{\varepsilon\log n}$ paths, by picking one endpoint from each of $|R_1|+|R_2|$ distinct paths and partitioning). Note that
$|S_{1}\setminus V(F)|\geq |S_{1}|-\varepsilon n\geq \varepsilon n$, since at
each previous step we add at most one vertex to $V(F)$ and Step~1 left $F$
with at most $\varepsilon n$ components. For any $W\subseteq S_{1}\setminus V(F)$
with $|W|\geq \varepsilon n/2$, we have
\[
|R_{i}|\cdot|W|\cdot p
\;\geq\;\frac{50n}{\varepsilon\log n}\cdot\frac{\varepsilon n}{2}\cdot\frac{\log n}{n}
\;=\;25n\;\geq\;10n,
\]
so by \ref{p:alpha joint} each $R_{i}$ has a neighbour in $W$. Hence each
$R_{i}$ is adjacent to more than half of $S_{1}\setminus V(F)$, and by
pigeonhole there is a vertex $w\in S_{1}\setminus V(F)$ adjacent to both
$R_{1}$ and $R_{2}$. Since the two endpoints picked up by $w$ come from
distinct paths (as $R_{1}\cup R_{2}$ has at most one endpoint per path),
joining $w$ via these two edges merges two paths into one,
reducing the component count by one.

\smallskip\noindent\textbf{Step 3: turn the linear forest into a single path,
using $S_{2}$.}
After Steps 1--2, we have a linear forest $F_{E}$ covering $F$ with
$V(F_{E})\subseteq V(F)\cup S_{1}$ and at most
$\tfrac{100n}{\varepsilon\log n}$ components. Order its paths as
$P_{1},\dots,P_{L}$, and write $a_{i},b_{i}$ for the endpoints of $P_{i}$.
Let $I$ be the set of all such endpoints, regarded as a subgraph with no
edges (so $d_{I}(x)=0$ for $x\in I$). Set $G':=G[I\cup S_{2}]$. We connect
$b_{i}$ to $a_{i+1}$ by vertex-disjoint paths in $G'$ via repeated application
of~\Cref{cor:3.12}, with parameters $m=\lfloor n/(2\log\log n)\rfloor$ and $D=\lfloor\log\log\log n\rfloor$.

We first check that $G'$ is $m$-joined. For any two disjoint sets $A,B\subseteq V(G')\subseteq V(G)$ with $|A|,|B|\geq m$, we have $|A|\cdot|B|\cdot p\geq m^{2}p\geq \frac{n^{2}}{4(\log\log n)^{2}}\cdot\frac{\log n}{n}=\frac{n\log n}{4(\log\log n)^{2}}\geq 10n$ for $n$ large, so by~\Cref{lem:gnp properties}\ref{p:alpha joint} there is an edge between $A$ and $B$ in $G$ (hence in $G'$, since $A,B\subseteq V(G')$).

\emph{$I$ is $(m,D)$-extendable in $G'$.}
We need to verify, for every $X\subseteq V(G')$ with $|X|\leq 2m$, that
\[
|(N_{G'}(X)\cup X)\setminus V(I)|\;\geq\;(D-1)|X|-\sum_{x\in X\cap V(I)}(d_{I}(x)-1)
\;=\;(D-1)|X|+|X\cap V(I)|.
\]
It suffices to show $|N_{S_{2}}(X)|\geq |X|D$, since $N_{S_{2}}(X)\cap V(I)=\emptyset$ implies $N_{S_{2}}(X)\subseteq (N_{G'}(X)\cup X)\setminus V(I)$, and $|X|D\geq (D-1)|X|+|X|\geq (D-1)|X|+|X\cap V(I)|$.

Fix such an $X$ and set $G^{*}:=G'[X\cup S_{2}]$. Note that
$V(G^{*})\setminus X = S_{2}\setminus X$, so $N_{G^{*}}(X)\subseteq S_{2}$ and
$|N_{S_{2}}(X)|=|N_{G^{*}}(X)|$. Every vertex of $V(G^{*})$ has at least
$10(\log\log n)^{3}\geq (\log\log n)^{3}$ neighbours in $S_{2}\subseteq V(G^{*})$;
since $|V(G^{*})|\geq |S_{2}|\geq Kn/4\geq \varepsilon n$ (using $\varepsilon\leq K/4$),
\Cref{lem:gnp properties}\ref{p:C-expander} applied to $V(G^{*})$ shows that
$G^{*}$ is a $\log\log n$-expander.

\begin{itemize}
  \item If $|X|\leq |V(G^{*})|/(2\log\log n)$, then the expansion property of
  $G^{*}$ gives $|N_{S_{2}}(X)|=|N_{G^{*}}(X)|\geq |X|\log\log n\geq |X|D$.
  \item If $|V(G^{*})|/(2\log\log n)<|X|\leq 2m$, then by the joining property
  of $G^{*}$, any subset of $V(G^{*})\setminus(X\cup N_{G^{*}}(X))$ of size at
  least $|V(G^{*})|/(2\log\log n)$ has an edge towards $X$ . Hence
  \[
  |N_{S_{2}}(X)|=|N_{G^{*}}(X)|
  \;\geq\;|V(G^{*})|-|X|-\frac{|V(G^{*})|}{2\log\log n}
  \;\geq\;\frac{Kn}{4}-\frac{2n}{\log\log n}\;\geq\;|X|D
  \]
  for $n$ sufficiently large (using $D=o(\log\log n)$ and $K>0$ fixed).
\end{itemize}

\emph{Iterating~\Cref{cor:3.12}.}
The number of $(b_{i},a_{i+1})$-connections to find is $L-1\leq \tfrac{100n}{\varepsilon\log n}$. At each step we invoke~\Cref{cor:3.12} with path-length parameter $\ell=2k+1$ where $k=\lceil\log(2m)/\log(D-1)\rceil$, so $\ell-2k-1=0$ and each path produced has length $\ell=2k+1=o(\log n)$. Thus the total number of internal vertices used over all steps is at most $L\cdot \ell=o(n)$, so at each step the current extendable subgraph has at most $|I|+o(n)=o(n)\leq |G'|-10\,Dm$ vertices (and hence at most $|G'|-10\,Dm-(\ell-2k-1)$ vertices), so~\Cref{cor:3.12} applies.

\smallskip\noindent\textbf{Step 4: extend the spanning path to a Hamilton cycle,
using $S_{3}$.}
Steps 1--3 produce a path $P\subseteq V(G)\setminus S_{3}$ that covers $E(F)$;
let $x,y$ be its endpoints, and write $W:=V(P)\setminus\{x,y\}$. The graph $G'':=G-W$ has vertex set $V(G)\setminus W\supseteq S_{3}\cup\{x,y\}$ (and may additionally contain vertices of $S_{1}\cup S_{2}$ that were not used in Steps~2--3). The endpoints $x,y$ remain endpoints of paths of the original forest $F$, since each of Steps~1--3 either adds edges between original endpoints or extends paths with internal vertices, never producing a new path endpoint; in particular the lemma's degree hypothesis applies to $x$ and $y$. Every other vertex $v\in V(G'')$ lies in $S=V(G)\setminus V(F)$, so the hypothesis also applies to $v$. Combined with the splitting at the start of the proof, this gives $|N_{G}(v)\cap S_{3}|\geq 10(\log\log n)^{3}$ for every $v\in V(G'')$, hence $v$ has at least $10(\log\log n)^{3}\geq (\log\log n)^{3}$ neighbours inside $V(G'')\supseteq S_{3}$. By \Cref{lem:gnp properties}\ref{p:C-expander}, $G''$ is a $\log\log n$-expander, hence is Hamilton-connected by~\Cref{thm:hamilton-connected}, and so contains a Hamilton path from $x$ to $y$. Uniting this path with $P$ gives a Hamilton cycle of $G$ covering $F$.
\end{proof}

\section{The sparse range}\label{sec:sparse-range}

In this section we prove our main result.

\begin{theorem}\label{thm:sparse-range}
Let $C>0$ and let $\omega(n)$ be any function tending to infinity arbitrarily slowly. Let $p$ be such that
\[
\frac{\log n+\log\log n+\omega(n)}{n}\leq p\leq \frac{C\log n}{n}.
\]
Then $G(n,p)$ has a tight Hamilton cover whp.
\end{theorem}

\begin{proof}
Consider $G\sim G(n,p)$ with $(\log n + \log\log n + \omega(n))/n\leq p\leq C\log n/n$. We will show that whp $G$ can be covered by
$\lceil\Delta(G)/2\rceil$ Hamilton cycles. We choose $t:=\lceil 24/\alpha\rceil$ below (with $\alpha$ defined in the next paragraph) and set $K:=1/(3t)$, both positive constants depending only on $C$. Henceforth we fix $G$ satisfying all
the properties of~\Cref{lem:maxdegree} and~\Cref{prop:bad vertices}. We additionally condition on two further whp events: $\delta(G)\geq 2$ (which holds since $p\geq (\log n+\log\log n+\omega(n))/n$, see e.g.~\cite{bollobas1998random}); and the event that, whenever $np\geq 1.01\log n$, no vertex of $G$ has degree less than $(\log\log n)^{4}$. For this last event, a Poisson-type Chernoff bound gives $\bP[d(v)<(\log\log n)^{4}]\leq (e\cdot np/(\log\log n)^{4})^{(\log\log n)^{4}}e^{-np}\leq n^{-1.01+o(1)}$, so a union bound yields probability $o(1)$.

Fix $\alpha:=\alpha(C)>0$ small enough so that $(\tfrac{1}{10}-\alpha)\sqrt{2/C}\geq 2\alpha$ holds (e.g.\ $\alpha:=1/(10^{4}\sqrt{C})$), and set $t:=\lceil 24/\alpha\rceil$ (so $t\geq 24/\alpha$ and the bound $t(1-2\alpha/3)/(t-2)\leq 1-\alpha/2$ required below holds). Let $\varepsilon$ be a constant with $\varepsilon\ll \min(\alpha, 1/C)$, and set $k:=\lceil\Delta(G)/2\rceil$.
Let $B$ denote the set of vertices $v$ with $d(v)\geq (1-\alpha)\Delta(G)$, and
let $S$ denote the set of vertices with $d(v)<(\log\log n)^{4}$. Note that for $n$ large, by~\Cref{lem:maxdegree} we have $\Delta(G)\geq pn+(1-o(1))\sqrt{2pn\log n}$, so
\begin{align*}
(1-\alpha)\Delta(G)&\geq (1-\alpha)pn+(1-\alpha-o(1))\sqrt{2pn\log n}\\
&\geq pn+(1-1/10)\sqrt{2pn\log n}
\end{align*}
holds for $n$ large. Hence $B$ is contained in the set $B_{\ref{prop:bad vertices}}$ from~\Cref{prop:bad vertices}; similarly $S=S_{\ref{prop:bad vertices}}$. All conclusions of~\Cref{prop:bad vertices} for $B_{\ref{prop:bad vertices}}, S_{\ref{prop:bad vertices}}$ therefore apply to $B,S$. Write
$W:=S\cup B\cup N(S)$.


\subsection{Decomposing $G-W$ into linear forests}

In the rest of the subsection, we find a collection $\cF_M$ of linear forests which cover \( G - W \) and have properties which are convenient for extending them later into Hamilton cycles.
The following lemma allows us to partition the edges of \( G - W \) into a constant number of subgraphs \( G_i \), ensuring that each subgraph has approximately the same maximum degree. Additionally, each vertex \( v \in V(G_i) \) retains a lot of neighbors in the set \( V(G) \setminus V(G_i) \).
\begin{lemma}\label{lem:linear forests with reservoirs}
The edges of the graph $G-W$ can be partitioned into subgraphs $G_1,\dots,G_t$ with the following properties
\begin{itemize}
    \item\label{partition max deg} $\Delta(G_i)\le (1-3\alpha/4)\frac{\Delta(G)}{t-2}$ for each $i\in [t]$. 
    \item Denote $R_i=V(G)\sm (W\cup V(G_i))$. Then $|R_i|\geq n/2t$. 
    \item Every vertex $v\in V(G)\setminus W$ satisfies $e_G(v,R_i)\geq d(v)/200t$
\end{itemize}
\end{lemma}

\begin{proof}
We first record that $d_{W}(v)\leq 1$ for every $v\in V(G)\setminus W$. Indeed, $v$ has at most one $B$-neighbour (else a length-$2$ path connecting vertices in $B$ contradicts~\Cref{prop:bad vertices}~item~2), at most one $S$-neighbour (similarly), and at most one $N(S)$-neighbour (two such $w_{1},w_{2}$ with distinct $S$-witnesses give a length-$4$ path within $S$; with the same witness $s$, the $4$-cycle $w_{1}\text{-}v\text{-}w_{2}\text{-}s\text{-}w_{1}$ contains $s\in S$, contradicting the cycle clause). Moreover, $v$ cannot have neighbours in two of the sets $B,S,N(S)$: an $S$- and a $B$-neighbour give a length-$2$ path between $B$ and $S$; an $N(S)$- and a $B$-neighbour give a length-$3$ path between $B$ and $S$; and an $S$-neighbour together with an $N(S)$-neighbour would force $v\in N(S)\subseteq W$. So $d_{W}(v)\leq 1$.

Partition $V(G)$ randomly into sets $R_1',\dots,R_t'$, where each vertex is assigned a part independently and uniformly at random. Define $R_i:=R_i'\setminus W$.
Colour every edge of $G[R_i']$ uniformly at random with one of the colours in $[t]\setminus\{i\}$. For every pair $1\le i<j\le t$, colour each edge of $G[R_i',R_j']$ uniformly at random with an element of $[t]\setminus\{i,j\}$. Denote by $G_i$ the graph on $V(G)\setminus (W\cup R_i)$ consisting of all edges of colour $i$ (with both endpoints in this vertex set).

Recalling that $|W|=O(\sqrt{n}\log n)$ by \Cref{prop:bad vertices}, by Chernoff we have that the size event $\mathcal{A}:=\{|R_i|\geq n/(2t)\text{ for every }i\}$ holds with probability $$1-O(t)\cdot e^{-\Theta(n)}=1-e^{-\Theta(n)}.$$
For each $v\in V(G)\setminus W$ and each $i\in[t]$, let $A(v,i)$ be the event that $e_G(v,R_i)< d_G(v)/(200t)$ and $B(v,i)$ the event that $d_{G_i}(v)>(1-3\alpha/4)\Delta(G)/(t-2)$. The conditional distribution of $e_G(v,R_i)$ given $v\notin R_i$ is $\mathrm{Bin}(d_G(v)-d_W(v),1/t)$ (since neighbours in $W$ are removed from $R_i$), with mean at least $(d_G(v)-1)/t\geq d_G(v)/(2t)$ for $n$ large (using $d_G(v)\geq(\log\log n)^{4}\to\infty$). By Chernoff,
$\Pr[A(v,i)]\leq e^{-\Omega(d_G(v)/t)}\leq e^{-\Omega((\log\log n)^{4}/t)}=o(1/\Delta(G)^{2})$ and similarly $\Pr[B(v,i)]\leq e^{-\Omega(\Delta(G)/t)}=o(1/\Delta(G)^{2})$ (using that the expected degree of every vertex is at most $(1-\alpha)\Delta(G)/t$). Each $A(v,i)$ and $B(v,i)$ depends only on the random choices at vertices of $N_{G}(v)\cup\{v\}$ and on the edges incident to $v$, hence is mutually independent of all but $O(t^{2}\Delta(G)^{2})=O(\log^{2}n)$ other such events. With $2nt$ such events in total and each occurring with probability $o(1/\Delta(G)^{2})= o(1/\log^{2}n)$, the symmetric Lov\'asz Local Lemma applies and gives
\[
\Pr[\mathcal{B}]\geq \exp\bigl(-2\cdot 2nt\cdot o(1/\log^{2}n)\bigr)=\exp(-o(n)),
\]
where $\mathcal{B}:=\{\text{no }A(v,i)\text{ or }B(v,i)\text{ holds}\}$. Since $\Pr[\mathcal{A}^{c}]\leq e^{-\Theta(n)}$, we have $\Pr[\mathcal{A}\cap\mathcal{B}]\geq \Pr[\mathcal{B}]-\Pr[\mathcal{A}^{c}]\geq \exp(-o(n))-e^{-\Theta(n)}>0$, so there is a partition meeting all three properties simultaneously.
\end{proof}

\noindent For each $i\in[t]$, applying~\Cref{thm:approxLAC} to $G_i$ gives a collection of at most $(1+\varepsilon) \frac{\Delta(G_i)}{2}$ linear forests covering~$G_i$, for any fixed constant $\varepsilon>0$ (we choose $\varepsilon\ll \alpha$ so that the bound below holds).

\begin{defin}\label{def:F_M}
For each $i \in [t]$, let $K:=\left\lceil\left(1 - \frac{2\alpha}{3}\right) \frac{\Delta(G)}{2(t-2)}\right\rceil$ and apply \Cref{thm:approxLAC} to $G_i$ (with $\varepsilon\ll \alpha$) to obtain a collection $\mathcal{F}_{G_i}$ of at most $K$ linear forests covering $G_i$; we pad with empty forests (containing no edges) if needed so that $|\mathcal{F}_{G_i}|=K$ exactly. We denote the union of these collections by $\cF_M$ and call it the \emph{Middle collection}. By our choice of $t$, we have $t(1-2\alpha/3)/(t-2)\leq 1-\alpha/2$, and thus
\[
|\cF_M|=t\cdot K\leq t\cdot\left(\left(1-\frac{2\alpha}{3}\right)\frac{\Delta(G)}{2(t-2)}+1\right)= \frac{t(1-2\alpha/3)}{t-2}\cdot \frac{\Delta(G)}{2}+t\leq \left(1-\frac{\alpha}{2}\right)\lceil\Delta(G)/2\rceil
\]
for $n$ large, since $\Delta(G)=\omega(t)$ absorbs the additive $t$. We may therefore assume that $\cF_M$ has precisely $M:=\left\lceil\left(1 - \frac{\alpha}{2}\right) \lceil\Delta(G)/2\rceil\right\rceil$ linear forests by adding (if needed) padding forests to $\cF_M$; each padding forest consists of a single edge of $G-W$, chosen to lie in some specific $G_i$ (so each padding forest is associated with a well-defined unique index $i\in[t]$). We choose these padding edges to be distinct so that every edge of \(G-W\)
  appears in at most two forests of \(\cF_M\).

Such padding edges exist whp, since we need only \(O(\Delta(G))=O(\log n)\)
  of them while \(G-W\) has \(\Omega(n\log n)\) edges. Each padding forest is
  the two-vertex graph whose edge set consists of one chosen padding edge.

\end{defin}

\subsection{Handling $S$ and $B$}

We start by covering the edges touching $B$ with linear forests.
\begin{lemma}
  Let \(k:=\lceil\Delta(G)/2\rceil\). There is a collection
  \(\mathcal F_B=\{F_1^B,\ldots,F_k^B\}\) of linear forests such that:
  \begin{enumerate}
  \item every edge incident to \(B\) is contained in at least one and at most
  two
  forests \(F_i^B\);
  \item for every \(b\in B\) and every \(i\in[k]\), we have
  \(d_{F_i^B}(b)=2\);
  \item all endpoints of the forests \(F_i^B\) lie outside \(B\).
  \end{enumerate}
\end{lemma}

\begin{proof}
  By~\Cref{prop:bad vertices}~Item~2, the edges incident to \(B\) form disjoint
  stars. Let \(k=\lceil\Delta(G)/2\rceil\). For each \(b\in B\), distribute the
  edges of its star among \(F_1^B,\ldots,F_k^B\) so that each forest gets two
  edges at \(b\) allowing an edge
  to be used twice if needed, but not more than twice; this is possible since \(k\le d_G(b)\le 2k\). Then each \(F_i^B\) is a disjoint union of
  cherries
  centred in \(B\), and the required properties are immediate.
  \end{proof}

Now we turn to edges touching $S$. Recall that those are small degree vertices, and the edges touching them necessarily must be contained in many Hamilton cycles on average in the final cover. Our goal here is to cover the edges touching $S\cup N(S)$ with linear forests, but so that edges between $N(S)$ and $N^2(S)$ are not used many times. 

\begin{lemma}
  Whp, for every \(v\in S\), there is a collection
  \(\mathcal F_v=\{F_{v,1},\ldots,F_{v,k}\}\), where
  \(k=\lceil\Delta(G)/2\rceil\), of linear forests contained in $G$ such that:
  \begin{enumerate}
  \item every edge in $G$ incident to a vertex of \(\{v\}\cup N(v)\) is contained in at
  least one forest \(F_{v,j}\);
  \item every edge in $G$ not incident to \(v\) is contained in at most \(100\) forests
  \(F_{v,j}\);
  \item  for every \(j\), every vertex of \(\{v\}\cup N(v)\) has degree \(2\) in
  \(F_{v,j}\);
  
  \item for every $j$ every vertex of
  \(V(F_{v,j})\setminus(\{v\}\cup N(v))\) has degree \(1\) in $F_{v,j}$ and lies in
  \(N^2(v)\).

  \end{enumerate}

\end{lemma}


\begin{proof}
We consider two regimes.

\textit{Case I: $np\leq 1.01\log n$.}
  By~\Cref{prop:bad vertices} item~3, every $x\in N(v)$ satisfies
  \[
  \log n/10\leq d(x)\leq 1.1\log n.
  \]
  Also, by~\Cref{lem:maxdegree}, whp
  \[
  1.2\log n\leq k:=\lceil\Delta(G)/2\rceil\leq 2\log n
  \]
  for $n$ large. Since $\delta(G)\geq 2$, write
  \[
  N(v)=\{x_1,\ldots,x_r\},
  \qquad 2\leq r=d(v)<(\log\log n)^4.
  \]
  For each $i\in[r]$, put $L_i:=N(x_i)\setminus\{v\}$. By
  \Cref{prop:bad vertices} item~2, the sets $L_i$ are pairwise disjoint subsets
  of $N^2(v)$: an edge inside $N(v)$ would give a triangle through $v$, and two
  vertices of $N(v)$ with a common neighbour other than $v$ would give a
  $4$-cycle through $v$.

  For each $j\in[k]$, we first choose the two neighbours of $v$ which will lie
  on the path through \(v\). Let
  \[
  A_j :=
  \begin{cases}
  \{x_1,x_2\}, & j\notin \{3,\ldots,r\},\\
  \{x_1,x_j\}, & 3\leq j\leq r.
  \end{cases}
  \]
  Thus the forests with $j=3,\ldots,r$ are responsible for covering the edges
  $vx_3,\ldots,vx_r$, while all other forests use the edge-pair $vx_1,vx_2$.

  Now form \(F_j\) as follows. If \(A_j=\{x_a,x_b\}\), put into \(F_j\) a path
  \[
  \ell_{a,j}\,x_a\,v\,x_b\,\ell_{b,j},
  \]
  where \(\ell_{a,j}\in L_a\) and \(\ell_{b,j}\in L_b\). For every
  \(x_i\in N(v)\setminus A_j\), put into \(F_j\) a cherry
  \[
  \ell_{i,j}\,x_i\,\ell'_{i,j},
  \]
  with distinct \(\ell_{i,j},\ell'_{i,j}\in L_i\).

  Notice that we still did not exactly specify the vertices $\ell_{*,*}$, but merely the set $L_*$ they belong to.
  It remains to choose these leaf-edges. For a fixed vertex \(x_i\), count the
  number \(s_i\) of slots across $F_j$'s requiring an edge from \(x_i\) to \(L_i\). From the
  definition of the sets \(A_j\),
  \[
  s_1=k,\qquad s_2=k+r-2,\qquad s_i=2(k-1)+1 \quad\text{for } i\geq 3.
  \]
  Let \(q_i:=|L_i|=d(x_i)-1\). Then
  \[
  \log n/10-1\leq q_i\leq 1.1\log n,
  \]
  while, for \(n\) large,
  \[
  q_i\leq s_i\leq 5\log n\leq 100q_i.
  \]
  Enumerate the edges from \(x_i\) to \(L_i\) cyclically and assign them to the
  \(s_i\) slots in the cyclic order, making the two slots of any cherry consecutive.
  Then every edge from \(x_i\) to \(L_i\) is used at least once and at most
  \(100\) times, and every cherry has two distinct leaves.

  Since the sets \(L_i\) are pairwise disjoint and lie in \(N^2(v)\), each
  \(F_j\) is a vertex-disjoint union of one path through \(v\) and cherries
  centred at the remaining vertices of \(N(v)\). Hence each \(F_j\) is a linear
  forest whose endpoints lie in \(N^2(v)\), and it uses no vertices outside
  \(\{v\}\cup N(v)\cup N^2(v)\). Moreover, every vertex of \(\{v\}\cup N(v)\) has degree \(2\) in every
  \(F_j\),
  and every vertex outside \(\{v\}\cup N(v)\) used by \(F_j\) has degree \(1\)
  and lies in \(N^2(v)\).

  Finally, every edge incident to \(v\) is covered: \(vx_1\) appears in every
  forest, \(vx_2\) appears in all forests with \(j\notin\{3,\ldots,r\}\), and
  each \(vx_i\) with \(i\geq 3\) appears in \(F_i\). All edges incident to
  \(N(v)\) but not to \(v\) are the edges from some \(x_i\) to \(L_i\), and
  these
  are covered by the cyclic assignment above, each with multiplicity at most
  \(100\). This completes Case I.

  \textit{Case II: $np\geq 1.01\log n$.}
In this regime we have conditioned (at the start of the proof of \Cref{thm:sparse-range}) on the whp event that no vertex of $G$ has degree less than $(\log\log n)^{4}$; hence $S=\emptyset$, and the conclusion is vacuous. This completes the proof.
 \end{proof}

Since we know that the vertices in $S$ are pairwise at distance at least $5$, the above lemma implies the following.

 \begin{cor}
  There exists a collection
  \(\mathcal F_S=\{F_1^S,\ldots,F_k^S\}\), where
  \(k=\lceil\Delta(G)/2\rceil\), of linear forests such that:
  \begin{enumerate}
  \item every edge incident to \(S\cup N(S)\) is contained in at least one
  forest
  \(F_i^S\);
  \item every edge between \(N(S)\) and \(N^2(S)\) is contained in at most
  \(100\)
  forests \(F_i^S\);
  \item every vertex of \(S\cup N(S)\) has degree \(2\) in every \(F_i^S\);
  \item every vertex of \(V(F_i^S)\setminus(S\cup N(S))\) has degree \(1\) and
  lies in \(N^2(S)\).
  \end{enumerate}
  \end{cor}


\begin{remark}
By~\Cref{prop:bad vertices}, $B$ and $S$ have no path of length at most $3$ between them, which implies $(B\cup N(B))\cap (S\cup N(S)\cup N^{2}(S))=\emptyset$, so $\cF_B$ and $\cF_S$ are vertex disjoint. Both contain precisely $\lceil \Delta/2\rceil$ linear forests, so we can arbitrarily pair up the forests in those collections to obtain a collection $\cF_{E}$ of $\lceil\Delta/2\rceil$ linear forests (the $E$ in the index stands for \emph{extremal}, as we cover the vertices of smallest and highest degrees).
\end{remark}

\subsection{Joining $\cF_M$ and $\cF_E$}


\noindent We will now merge $\cF_M$ with $\mathcal{F}_E$  in order to create one collection $\mathcal{F}$ of $ \lceil\Delta(G)/2\rceil$ linear forests covering the edges of $G$, such that each endpoint in each forest has a large neighbourhood outside of that forest in~$G$. In the end, we will transform each of these linear forests into a Hamilton cycle using Lemma \ref{lem:linforestextension}, thus completing the proof. The construction of $\mathcal{F}$ proceeds as follows. Recall the notation $W=B\cup S\cup N(S)$.

First, recall that by Definition~\ref{def:F_M}, we have $|\cF_M| = M:=\lceil(1-\alpha/2) \lceil\Delta(G)/2\rceil\rceil$. Take a subcollection $\mathcal{F}_E' \subseteq \mathcal{F}_E$ of size $M$, and fix a bijection $\phi:\cF_M\to\cF_E'$.
 For each pair $(F_M,F_E)$ where $F_E=\phi(F_M)$, we can define the linear forest $\hat{F} := (F_M \cup F_E) \setminus Q$, where $Q$ is the set of edges in $F_M$ which are incident to vertices also incident to some edge of $F_E$. We now add all these new linear forests $\hat{F}$ to $\mathcal{F}$, so that currently, $|\mathcal{F}| = M$. Let \( G' \subseteq G \) be the subgraph of \( G \) consisting of the edges in \( G - W\) that are not part of any linear forest \( \hat{F} \in \mathcal{F} \); specifically, these are the edges belonging to some set \( Q \) as defined earlier. Then, the following properties hold.

\begin{claim}
$\Delta(G') \leq 500$.
\end{claim}

\begin{proof}
First, $d_{G'}(v)=0$ for all $v\in W$, as the edges in $Q$ come from forests in $\cF_M$ with vertices in $G-W$. So fix $x\in V(G)\setminus W$.
If an edge $xy$ belongs to $G'$, then there is a pair $(F_M,F_E)$ with $xy \in Q\cap F_M$. There are two reasons $xy$ can be in $Q$: either $y$ is incident to an edge of $F_E$, or $x$ is.

The former type appears at most once for a fixed $x\in V(G)\setminus W$. To see this, note that $y$ lies in $N(W)\setminus W=N(B)\cup N^{2}(S)$ (since edges of $F_E$ incident to $V(G)\setminus W$ have their $V(G)\setminus W$ endpoint in $N(W)\setminus W$). We have $d_G(x, N(B)\cup N^{2}(S))\leq 1$ by \Cref{rem:bad vertices low degree}. This graph edge may occur in two forests of
  \(\cF_M\), because of the padding forests (see \Cref{def:F_M}), and so we count it with
  multiplicity at most two.

For the latter type, $x$ being incident to an edge of some $F_E\in\cF_E$ means $x\in N(W)\setminus W = N(B)\cup N^2(S)$, and $x$'s neighbour along that edge lies in $W$; again $d_G(x,W)\leq 1$ gives at most one such edge of~$G$, and since each edge appears in at most $100$ forests of $\cF_E$, $x$ lies in at most $100$ such forests. For each such forest, at most $2$ edges incident to $x$ in $F_M$ can land in $Q$. Together this gives $d_{G'}(x)\leq 1 + 100\cdot 2 = 201$, so $\Delta(G')\le 500$, with room to spare.
\end{proof}

\noindent Let $G'_i := G' \cap G_i$ for each $1 \leq i \leq t$ with $G_i$'s from \Cref{lem:linear forests with reservoirs}. We will now add linear forests to $\mathcal{F}$ which cover the edges in $G'$, as well as the edges used in $\mathcal{F}_E'' := \mathcal{F}_E \setminus \mathcal{F}_E'$. We do this so that at the end we have $|\mathcal{F}| \leq  \lceil\Delta(G)/2\rceil$ as desired.
Let us partition $\mathcal{F}_E''$ arbitrarily into $t$ collections $\mathcal{F}_E''(1), \ldots, \mathcal{F}_E''(t)$ each of size at least $\lfloor|\mathcal{F}_E''|/t\rfloor\geq \lfloor(\lceil\Delta(G)/2\rceil-M)/t\rfloor$, which exceeds $4\cdot 100\cdot \max\{1,\Delta(G')\}+1$ (since $\lceil\Delta/2\rceil-M\geq (\alpha/2)\lceil\Delta/2\rceil-O(1)$, $t=t(C)$ is constant, and $\Delta(G')\leq 500$). For each $i \leq t$, we apply \Cref{lem:merginglinfor} with $H_2:= G'_i$ and the collection $\cF:=\cF_E''(i)$, and with $H_1$ defined as the subgraph of $G$ consisting of the edges of $\cF_E''(i)$.

Every vertex in $v\in V(H_2)\subseteq V(G)\setminus W$ has at most one neighbour in $H_1\subset W\cup N_G(W)$ by \Cref{rem:bad vertices low degree}. Furthermore, trivially each $v\in V(H_2)$ has degree at most $\Delta(G')\leq 500$ in $H_2$ by Claim~1.

By the construction of \(\cF_S\) and \(\cF_B\), every edge of \(H_1\) incident
  to \(V(H_2)\) appears in at most \(\ell:=100\) forests of
  \(\cF_E''(i)\): such an edge is either a \(B\)-\(N(B)\) edge, used at most
  twice, or an \(N(S)\)-\(N^2(S)\) edge, used at most \(100\) times.
 Setting $d:=\max\{1,\Delta(G')\}$, the hypothesis $|\cF_E''(i)|\geq 4d\ell+1$ holds for $n$ large. So Lemma \ref{lem:merginglinfor} finds a collection $\mathcal{F}(i)$ of $|\mathcal{F}_E'' (i)|$ linear forests covering the edges of $G'_i$ and those used in $\mathcal{F}_E'' (i)$. We add all the collections $\mathcal{F}(i)$ to $\mathcal{F}$; then $|\mathcal{F}| = M+|\cF_E''|\leq \lceil\Delta(G)/2\rceil$ and the linear forests in $\mathcal{F}$ cover all edges of $G$: the $\hat F$'s cover $\cF_E'$ together with all $\cF_M$-edges except those in $G'$, while the $\cF(i)$'s cover $G'$ together with $\cF_E''$. 

\begin{lemma}\label{lem:forests are good}
    With $K:=1/(3t)$, every $F \in \mathcal{F}$ satisfies the conditions of \Cref{lem:linforestextension}: $|V(G) \setminus V(F)|\geq Kn$, and $|N_G(v) \setminus V(F)| \ge 100(\log\log n)^3$ for every $v \in V(G)\setminus W$ (which includes every endpoint of a path in~$F$).
\end{lemma}

\begin{proof}
Set $K:=1/(3t)$, a positive constant depending on $C$ (since $t=t(C)$). First, recall that each $F \in \mathcal{F}$ is such that $F[V(G) \setminus W] \subseteq G_i$ for some $i\in[t]$ (the padding forests in $\cF_M$ were chosen to be single edges of $G-W$, each of which lies in a unique $G_i$ that we record; the $\hat F$'s inherit the index from their $F_M$ component; the $\cF(i)$'s use index $i$). So $R_i\setminus (N(B)\cup N^2(S)) \subseteq V(G) \setminus V(F)$. Further, by \Cref{lem:linear forests with reservoirs} we have $|R_i|\geq n/(2t)$, by \Cref{prop:bad vertices} we have $|B|,|S|\leq \sqrt{n}$, and by \Cref{lem:maxdegree} we have $\Delta(G)=O(\log n)$. Combining these,
\[
|V(G) \setminus V(F)| \;\geq\; |R_i\setminus (N(B)\cup N^2(S))| \;\geq\; n/(2t)-O(\log^{2} n)(|B|+|S|) \;\geq\; n/(2t)-o(n) \;\geq\; Kn,
\]
using $K=1/(3t)<1/(2t)$. The hypothesis $|V(G)\setminus V(F)|\geq Kn$ of \Cref{lem:linforestextension} therefore holds.
For the second part, fix $v\in V(G)\setminus W$ (every endpoint of a path of $F$ lies in $V(G)\setminus W$: by construction the endpoints of $\hat F$ and of $\cF(i)$-forests lie there, since vertices of $W$ are internal in every $\cF_E$-forest and $\cF_M$-forests are contained in $G-W$). By \Cref{rem:bad vertices low degree}, for every $v\notin W$:
\[
e_G\bigl(v,R_i\setminus (N(B)\cup N^{2}(S))\bigr)\;\geq\; e_G(v,R_i)-1\;\geq\; d(v)/(200t)-1.
\]
Since $v\notin S$ we have $d(v)\geq(\log\log n)^{4}$, and $t=t(C)$ is constant, so for $n$ sufficiently large $d(v)/(200t)-1\geq 100(\log\log n)^{3}$.
Therefore $|N_G(v)\setminus V(F)|\geq e_G(v,R_i\setminus (N(B)\cup N^{2}(S)))\geq 100(\log\log n)^{3}$.
\end{proof}

\subsection{Extending forests in $\mathcal{F}$ into Hamilton cycles}

\noindent To finish the proof, we want to extend the collection $\mathcal{F}$ of at most $\lceil\Delta(G)/2\rceil$ linear forests covering the edges of $G$ into a Hamilton cover. Each $F\in\mathcal F$ is nonempty: the padding forests are single edges of $G-W$, each $\hat F$ contains all of its associated $F_E$-edges (which are nonempty as $F_E$ pairs an $\cF_B$-cherry with an $\cF_S$-component), and each $\cF(i)$-forest contains the edges of some $F\in\cF_E''(i)$. We may further assume each $F\in\mathcal F$ has no isolated vertices by removing such vertices from $V(F)$; this changes $S:=V(G)\setminus V(F)$ only by enlarging it, so the hypotheses of~\Cref{lem:forests are good} are preserved. Then \Cref{lem:forests are good} shows that each $F\in \mathcal{F}$ satisfies the conditions of \Cref{lem:linforestextension}, so we can apply this lemma to extend each linear forest to a Hamilton cycle in $G$. This gives a covering of $G$ with $\lceil \Delta(G)/2 \rceil$ Hamilton cycles, as required.
\end{proof}

\section{The very dense range}\label{sec:very-dense}

  In this section we close the remaining gap near \(p=1\). We write
  \(q=1-p\). Recall that $q\le n^{-1/8}$, and $qn^2\rightarrow\infty$;

  We shall use the following Hamilton decomposition theorem
  for dense regular graphs due to K\"uhn and Osthus~\cite{kuhn2014hamilton}.

  \begin{theorem}\label{thm:dense-ham-decomp}
  For every \(\gamma>0\) there exists \(n_0\) such that every even-regular\footnote{A graph is even-regular if it is $d$-regular for an even integer $d$.}
  graph \(J\) on \(n\ge n_0\) vertices with
  $
  \delta(J)\ge (1/2+\gamma)n
  $
  has a Hamilton decomposition.
  \end{theorem}

It is enough to prove the following.
  \begin{theorem}\label{thm:very-dense-range}
  Let \(\eps>0\), and let \(q=q(n)\) satisfy
 $
  q\le n^{-\eps}
  $ and $
  qn^2\to\infty .
  $
  Then whp \(G(n,1-q)\) has a tight Hamilton cover.
  \end{theorem}

  We first regularise \(G\) by adding extra copies of edges already present in
  \(G\).

  \begin{lemma}\label{lem:very-dense-regularisation}
  Let \(G\sim G(n,1-q)\), where \(q\le n^{-\eps}\) and \(qn^2\to\infty\).
  Set
  $
  k:=\left\lceil \frac{\Delta(G)}2\right\rceil .
  $
  Then whp there exists a multigraph \(R\), with every edge of \(R\) lying in
  \(E(G)\), such that
  \[
  d_R(v)=2k-d_G(v)
  \qquad\text{for every }v\in V(G).
  \]
  Moreover, \(\Delta(R)=o(n)\).
  \end{lemma}

  \begin{proof}
  Let \(H:=\overline G\), and write \(\delta:=\delta(H)\) and
  \(\Delta:=\Delta(H)\). Since
  $
  \Delta(G)=n-1-\delta,
  $
  we have
  \[
  2k=n-1-\delta+\xi
  \]
  for some \(\xi\in\{0,1\}\). Thus the required degree in \(R\) is
  \[
  b(v):=2k-d_G(v)=d_H(v)-\delta+\xi .
  \]
  Let
  $
  B:=\sum_{v\in V(G)} b(v).
  $
  Notice that \(B\) is even, since
  $
  B=2kn-2e(G).
  $
  We shall use the following standard estimates for \(H\sim G(n,q)\). Since
  \(q=o(1)\) and \(qn^2\to\infty\), whp
  $
  \Delta=o(n)
  $
  and
  \[
  (\Delta+1)(\Delta-\delta+1)=o(B). \tag{1}
  \]
  Indeed, if \(nq\le 2\log n\), then whp \(B=\Omega(e(H))\) and
  \(\Delta^2=o(e(H))\). If \(nq>2\log n\), then the usual estimates for the
  extreme degrees of \(G(n,q)\) whp give
  $
  \Delta-\delta=O(\sqrt{nq\log n})
  $
  and
  $
  B\ge c n\sqrt{nq\log n}
$
  for some absolute constant \(c>0\), while
  \(\Delta=O(nq)\). Hence
  \[
  (\Delta+1)(\Delta-\delta+1)
  =O(nq\sqrt{nq\log n})
  =o(n\sqrt{nq\log n})
  =o(B),
  \]
  as \(q=o(1)\).

  Now create \(b(v)\) labelled stubs at each vertex \(v\). Define an auxiliary
  graph \(\Gamma\) on these \(B\) stubs by joining two stubs if they belong to
  distinct vertices \(u,v\) with \(uv\in E(G)\). A perfect matching in
  \(\Gamma\) gives precisely the desired multigraph \(R\).

  A stub at \(v\) is non-adjacent only to the other stubs at \(v\), and to stubs
  at vertices \(u\in N_H(v)\). Therefore its degree in \(\Gamma\) is at least
  \[
  B-b(v)-\sum_{u\in N_H(v)} b(u)
  \ge B-(\Delta+1)(\Delta-\delta+1)
  > B/2
  \]
  by (1). Hence \(\Gamma\) has a perfect matching, for instance by Dirac's
  theorem. This matching defines \(R\). Finally, whp we have
  \[
  \Delta(R)=\max_v b(v)\le \Delta-\delta+1=o(n).
  \]
  \end{proof}

  In particular, for $R$ given in the previous lemma, we know that $G+R$ is a $2k$-regular multigraph.
  We also need the following fact which allows us to complete matchings to Hamilton cycles in very dense graphs, and which follows for example from Theorem~1 in~\cite{posa1963circuits}

  \begin{lemma}\label{lem:prescribed-matching-cycle}
  Let \(J\) be a graph on \(n\) vertices with
  $
  \delta(J)\ge n-s,
  $
  where \(s=o(n)\). Let \(M\) be a matching on \(V(J)\). Then, for \(n\) large, there is a Hamilton
  cycle in $J\cup M$
  containing every edge of \(M\).
  \end{lemma}


Now we are ready for the proof of the main result of this section. 
  \begin{proof}[Proof of \Cref{thm:very-dense-range}]
  Let \(G\sim G(n,1-q)\), and let \(R\) be the multigraph supplied by
  \Cref{lem:very-dense-regularisation}. By Shannon's theorem, the edges of \(R\)
  can be properly coloured with
  $
  r\le \frac32\Delta(R)=o(n)
  $
  colours. Thus
  $
  E(R)=M_1\cup\cdots\cup M_r,
  $
  where each \(M_i\) is a matching in $G$.

  We cover the edges of the regular multigraph $G+R$, one Hamilton cycle at a time: at step $i$, the Hamilton cycle contains the matching $M_i$. Start with
  \(J_0:=G\). Suppose \(J_{i-1}\) has been defined. Since
  \[
  \delta(J_{i-1})\ge \delta(G)-2r
  \ge n-1-\Delta(\overline G)-2r
  = n-o(n),
  \]
  \Cref{lem:prescribed-matching-cycle} gives a Hamilton cycle \(C_i\) containing
  all edges of \(M_i\), with all other edges lying in \(J_{i-1}\). We regard the
  edges of \(M_i\) as extra edges coming from \(R\), and delete from \(J_{i-
  1}\)
  only the remaining edges of \(C_i\). Let the resulting graph be \(J_i\).

  After all \(r\) steps, for every vertex \(v\),
  \[
  d_{J_r}(v)
  =d_G(v)-2r+d_R(v)
  =d_G(v)-2r+2k-d_G(v)
  =2k-2r.
  \]
  Thus \(J_r\) is an even-regular simple graph. Moreover,
  $
  2k-2r=n-o(n),
  $
  so for \(n\) large we have \(\delta(J_r)\ge 2n/3\). By
  \Cref{thm:dense-ham-decomp}, \(J_r\) decomposes into \(k-r\) Hamilton cycles.

  Together with \(C_1,\ldots,C_r\), these cycles form a collection of exactly
  \(k=\lceil\Delta(G)/2\rceil\) Hamilton cycles. They cover every edge of \(G\):
  the decomposition of \(J_r\) covers all base edges not used during the peeling
  process, while the cycles \(C_i\) cover the remaining base edges and use the
  edges of \(R\) only as repeated copies. Hence \(G\) has a tight Hamilton
  cover.
  \end{proof}

  \begin{remark}
  The condition \(qn^2\to\infty\) is the natural endpoint for a whp statement.
  If \(q=c/n^2\) and \(n\) is odd, then with probability bounded away from zero
  the complement of \(G(n,1-q)\) consists of exactly one edge, say \(xy\). Then
  \(G=K_n-xy\) has no tight Hamilton cover. Indeed,
  \(k=\lceil\Delta(G)/2\rceil=(n-1)/2\), so \(k\) Hamilton cycles altogether have only one
  more edge-occurrence than \(|E(G)|\). But the two vertices \(x,y\) both have
  odd degree in \(G\), and each would need a repeated incident edge; since
  \(xy\notin E(G)\), these repeated edges must be distinct, a contradiction.
  \end{remark}

\section{Linear arboricity}\label{sec:linear-arboricity}

  Recall that \(\la(G)\) denotes the minimum number of linear forests whose
  union covers \(E(G)\).

  \begin{proof}[Proof of \Cref{thm:random-linear-arboricity}]
  We split into ranges of \(p\).

  First suppose \(p\geq 2\log n/n\) and \(1-p\gg n^{-2}\). In this range the
  Hamilton-cover result gives the desired conclusion by the standard deduction
  from Hamilton covers to linear arboricity; see, for instance, the proof of the
  linear-arboricity corollary in~\cite{GKO:16}. 

  Next suppose \(1-p=O(n^{-2})\). Then whp \(\overline G\) has $o(n)$ edges,
  and in particular \(G\) has a vertex of degree \(n-1\). Hence
  \(\Delta(G)=n-1\). Furthermore, by the standard Walecki construction
\cite{alspach2008walecki} the linear arboricity of \(K_n\)
is \(\lceil n/2\rceil\), meaning that \(K_n\) decomposes into \(\lceil n/2\rceil\) linear forests. Intersecting these forests with \(E(G)\) gives
  \(\lceil n/2\rceil=\lceil(\Delta(G)+1)/2\rceil\) linear forests covering
  \(G\).

  Now we consider \(c/n\leq p\leq 2\log n/n\) for any fixed $c>0$. Fix a sufficiently
  small constant \(\alpha>0\), put \(k:=\lceil\Delta(G)/2\rceil\), and let
  \(B\) be the set of vertices of degree at least \((1-\alpha)\Delta(G)\).
  By a high-degree separation result, similar to \Cref{prop:bad vertices}, whp \(B\) is independent, the stars from \(B\) are vertex-disjoint,
  \(N(B)\) is independent and each $v\notin B$ has at most one neighbour in $N(B)$ (indeed, such a result follows from a simple first moment argument). We condition on these properties.

  We first cover the edges incident to \(B\). For each \(b\in B\), enumerate
  \(N(b)=\{x_1^{(b)},\ldots,x_{d(b)}^{(b)}\}\). For \(j\in[k]\), take 
  the cherry at \(b\) using the two edges \(bx_{2j-1}^{(b)}\) and \(bx_{2j}^{(b)}\), with
  indices read cyclically modulo \(d(b)\). Taking the union over all \(b\in B\)
  gives linear forests \(F_1^B,\ldots,F_k^B\). These forests cover every edge
  incident to \(B\), and every such edge is used at most twice.

  Let \(G_0\) be obtained from \(G\) by deleting all edges incident to \(B\).
  Since
  every vertex outside \(B\) has degree less than \((1-\alpha)\Delta(G)\), we
  have
  \(\Delta(G_0)\leq (1-\alpha)\Delta(G)\). Applying the approximate linear
  arboricity theorem (\Cref{thm:approxLAC}) with a small enough error term, we decompose \(G_0\) into
  \(m\leq (1-\alpha/3)k\) linear forests \(A_1,\ldots,A_m\).

  For each \(i\in[m]\), pair \(A_i\) with \(F_i^B\). Let \(Q_i\) be the set of
  edges of \(A_i\) incident to an endpoint of \(F_i^B\), and define
  \(\widehat F_i:=F_i^B\cup(A_i\setminus Q_i)\). This is a linear forest: after
  removing \(Q_i\), no endpoint of \(F_i^B\) is incident to an edge of \(A_i\).

  It remains to cover \(Q:=\bigcup_{i\leq m}Q_i\), using the unused forests
  \(F_{m+1}^B,\ldots,F_k^B\). We first note that \(\Delta(Q)\leq 4\). Indeed, a
  vertex \(x\in N(B)\) is an endpoint in at most two forests \(F_i^B\), because
  the unique edge from \(x\) to \(B\) is used at most twice; in each
  corresponding
  forest \(A_i\), at most two edges incident to \(x\) are deleted. Thus
  \(d_Q(x)\leq 4\). Every vertex outside \(N(B)\) has at most one neighbour in
  \(N(B)\), by the same short path/cycle exclusion, and hence has degree at most
  one in \(Q\).

 We now apply \Cref{lem:merginglinfor}. Let
  \[
  \mathcal F_{\mathrm{rem}}:=\{F_{m+1}^B,\ldots,F_k^B\},
  \qquad
  H_1:=\bigcup_{F\in\mathcal F_{\mathrm{rem}}}F,
  \qquad
  H_2:=Q.
  \]
  We apply the lemma with \(\mathcal F=\mathcal F_{\mathrm{rem}}\),
  \(d=4\), and \(\ell=2\).

  Let us verify its hypotheses. The collection \(\mathcal F_{\mathrm{rem}}\)
  covers \(E(H_1)\) by definition, and every edge of \(H_1\) appears in at most two
  forests of \(\mathcal F_{\mathrm{rem}}\), since every edge incident to \(B\)
  appears in at most two forests \(F_j^B\). Also, by the preceding paragraph,
  \(\Delta(H_2)=\Delta(Q)\leq 4\).

  It remains to check the degree condition for \(H_1\) at vertices of
  \(V(H_2)\).
  If \(x\in V(Q)\cap N(B)\), then \(x\) has a unique neighbour in \(B\), because
  the stars from \(B\) are vertex-disjoint. Hence \(d_{H_1}(x)\leq 1\). If
  \(x\in V(Q)\setminus N(B)\), then \(x\notin B\) and \(x\) is incident to no
  edge of \(H_1\), since \(H_1\) consists only of edges between \(B\) and
  \(N(B)\). Thus \(d_{H_1}(x)\leq 1\) for every \(x\in V(H_2)\). Therefore every
  vertex of \(V(H_2)\) has degree at most \(4\) in both \(H_1\) and \(H_2\).

  Finally,
  $
  |\mathcal F_{\mathrm{rem}}|=k-m\geq \alpha k/3\to\infty,
  $
  so for \(n\) large
  $
  |\mathcal F_{\mathrm{rem}}|\geq 4d\ell+1=33.
  $
  Thus \Cref{lem:merginglinfor} gives \(k-m\) linear forests covering
  \(H_1\cup Q\). Together with
  \(\widehat F_1,\ldots,\widehat F_m\), this gives \(k\) linear forests covering
  all edges of \(G\).

For the remaining regime, suppose \(p=o(1/n)\). Then whp \(G\) is a forest. Greedily edge-colour \(G\) using colours
\([\,\Delta(G)\,]\). Pair the colour classes as
\(\{1,2\},\{3,4\},\ldots\). The union of two colour classes
has maximum degree at most two and is acyclic, hence is a linear forest. Thus \(G\) decomposes into at most \(\lceil\Delta(G)/2\rceil\) linear
forests.
\end{proof}
  %




\section{Concluding remarks}\label{sec:concluding}

We close by explaining why the sparse-range proof also gives
\Cref{thm:hitting-time-cover}. The only extra point is that
\(\tau_2\) is a random time, while the estimates in \Cref{sec:sparse-range} are
stated for a fixed binomial random graph. 

Let \(\xi=\xi(n)\to\infty\) grow sufficiently slowly, and put
\(m_{\pm}=\frac n2(\log n+\log\log n\pm \xi)\). By the standard hitting-time
theorem for minimum degree two, whp \(m_-\le \tau_2\le m_+\). We claim that,
whp, every graph \(G_m\) with \(m_-\le m\le m_+\) satisfies all deterministic
hypotheses used in the proof of \Cref{thm:sparse-range}, except possibly the
condition \(\delta(G_m)\ge2\).

Indeed, throughout this interval the edge density is \(\Theta(\log n/n)\). The
estimates used in \Cref{sec:sparse-range} are Chernoff-type degree estimates,
expansion and joining estimates, and first-moment bounds excluding short local
configurations around exceptional vertices. The same estimates hold uniformly
in the random graph process. To see this, put \(G^-=G_{m_-}\) and
\(G^+=G_{m_+}\). Since \(G^-\subseteq G_m\subseteq G^+\) for every
\(m\in[m_-,m_+]\), monotone lower-bound properties are inherited from \(G^-\),
and monotone upper-bound or exclusion properties are checked in \(G^+\). For
the non-monotone exceptional sets, we control fixed supersets: the low-degree
set at any time is contained in the low-degree set of \(G^-\), while the
high-degree set at any time is contained, after slightly weakening constants,
in a high-degree set defined using \(G^+\). The usual first-moment estimates
show that these fixed sets are small and well separated in \(G^+\). Thus the
deterministic hypotheses of the sparse construction hold simultaneously
throughout the window.

Now consider \(G_{\tau_2}\). On the high-probability event addressed above,
\(G_{\tau_2}\) satisfies all deterministic hypotheses of the sparse construction,
and the remaining hypothesis \(\delta(G_{\tau_2})\ge2\) holds by definition of
\(\tau_2\). Therefore the construction from \Cref{sec:sparse-range} applies
verbatim to \(G_{\tau_2}\): it produces exactly
\(\lceil\Delta(G_{\tau_2})/2\rceil\) linear forests covering all edges of
\(G_{\tau_2}\), and then extends them to Hamilton cycles using the reserved
random structure. This gives a tight Hamilton cover of \(G_{\tau_2}\) whp. 

\vspace{0.5 cm}

\textbf{Acknowledgment.} The authors used ChatGPT for language polishing, literature-search assistance, and informal discussion of minor auxiliary details. The main results and proofs are due to the authors.

\bibliographystyle{plain}

\appendix

\end{document}